\documentclass[11pt]{article}

\usepackage[margin=1in]{geometry}
\usepackage[T1]{fontenc}

\usepackage[utf8]{inputenc}
\usepackage{amsmath}
\usepackage{amsthm}
\usepackage[round]{natbib}  
\usepackage{babel}
\usepackage{graphicx}
\usepackage{url}
\usepackage{subcaption}
\usepackage{booktabs}
\usepackage{multirow}
\usepackage{placeins}
\usepackage{amssymb}
\usepackage{multicol}
\usepackage{algorithm}
\usepackage{algcompatible}
\usepackage{enumitem}
\usepackage{setspace}
\usepackage{bbm}
\usepackage{accents}
\newcommand*{\dt}[1]{%
  \accentset{\mbox{\large\bfseries .}}{#1}}

\setlist[itemize]{itemsep=0pt,parsep=1pt,topsep=1pt,partopsep=1pt,leftmargin=*}
\setlist[enumerate]{itemsep=0pt,parsep=1pt,topsep=1pt,partopsep=1pt, leftmargin=*}
\setlist[description]{itemsep=0pt,parsep=1pt,topsep=1pt,partopsep=1pt, leftmargin=*}

\usepackage[margin=1in]{geometry}

\normalsize

\usepackage{titling}

\usepackage{titlesec}

\usepackage{hyperref}
\hypersetup{
    colorlinks=true,
    linkcolor=blue,
    urlcolor=blue,
    citecolor=blue,
}

\providecommand{\theoremname}{Theorem}
\providecommand{\propositionname}{Proposition}
\providecommand{\corollaryname}{Corollary}
\providecommand{\lemmaname}{Lemma}
\providecommand{\remarkname}{Remark}
\providecommand{\proofname}{Proof}

\newcommand{\Lcal}{\mathcal{L}}
\newcommand{\Scal}{\mathcal{S}}

\newcommand{\softmax}{\operatorname{SoftMax}}

\newcommand{\topkind}{\operatorname{TopKInd}}
\newcommand{\sign}{\operatorname{Sign}}
\newcommand{\newinds}{\dt{\alpha}}
\newcommand{\rmvinds}{\newinds_{-}}
\newcommand{\addinds}{\newinds_{+}}
\newcommand{\bijection}{\mathfrak{b}}
\newcommand{\allpairs}{\mathfrak{B}}

\newcommand{\E}{\mathbb{E}}
\newcommand{\Var}{\operatorname{Var}}

\newcommand{\R}{\mathbb{R}}

\newcommand{\calZ}{\mathcal{Z}}
\newcommand{\proj}{\operatorname{Proj}}
\newcommand{\ind}{\mathbbm{1}}
\newcommand{\ones}{\mathbf{1}}

\newcommand{\gap}{\mathfrak{G}}
\newcommand{\ovrload}{\mathbf{U}_{>}}
\newcommand{\balload}{\mathbf{U}_{=}}
\newcommand{\udrload}{\mathbf{U}_{<}}
\newcommand{\lr}{\texttt{lr}}
\newcommand{\wdy}{\texttt{wd}}
\newcommand{\bi}{\bar{i}}

\theoremstyle{plain}
\newtheorem{thm}{\protect\theoremname}
\theoremstyle{plain}
\newtheorem{lem}[thm]{\protect\lemmaname}
\ifx\proof\undefined
\newenvironment{proof}[1][\protect\proofname]{\par
	\normalfont\topsep6\p@\@plus6\p@\relax
	\trivlist
	\itemindent\parindent
	\item[\hskip\labelsep\scshape #1]\ignorespaces
}{%
	\endtrivlist\@endpefalse
}
\providecommand{\proofname}{Proof}
\fi
\theoremstyle{plain}

\theoremstyle{remark}
\newtheorem{rem}[thm]{\protect\remarkname}
\theoremstyle{plain}
\newtheorem{prop}[thm]{\protect\propositionname}

\title{
    \textbf{A Theoretical Framework for Auxiliary-Loss-Free Load Balancing of Sparse Mixture-of-Experts \\ in Large-Scale AI Models}
}
\author{
    X.Y. Han\thanks{Authors listed alphabetically.} \\
    Chicago Booth \\
    \texttt{XY.Han@chicagobooth.edu}
    \and 
    Yuan Zhong\footnotemark[1] \\
    Chicago Booth \\
    \texttt{Yuan.Zhong@chicagobooth.edu}
}

\allowdisplaybreaks

\begin{document}
\onehalfspacing
\maketitle

\begin{abstract}
\noindent In large-scale AI training, Sparse Mixture-of-Experts (s-MoE) layers enable scaling by activating only a small subset of experts per token. An operational challenge in this design is load balancing: routing tokens to minimize the number of idle experts, which is important for the efficient utilization of costly GPUs and for the thorough training of architecture parameters across all experts. We provide a theoretical framework for analyzing the Auxiliary-Loss-Free Load Balancing (ALF-LB) procedure --- proposed by DeepSeek's \citet{wang2024auxiliary} --- by casting it as a primal-dual method using a single-shot, constant-time update per training iteration for solving an assignment problem. First, in a stylized deterministic setting, our framework yields several insightful structural properties: (i) a monotonic improvement condition for the Lagrangian objective, (ii) a preference rule that moves tokens from overloaded to underloaded experts, and (iii) an approximate-balancing guarantee. Then, we incorporate the stochastic and dynamic nature of AI training using a generalized online optimization formulation. In the online setting, we derive a strong convexity property of the objective that leads to a logarithmic expected regret bound under certain step-size choices. Additionally, we present real experiments on 1B-parameter DeepSeekMoE models to complement our theoretical findings. Together, these results build a principled framework for analyzing the Auxiliary-Loss-Free Load Balancing of s-MoE in AI models.\sloppy
\end{abstract}

\section{Introduction: s-MoEs and Load Balancing in AI Training}\label{sec:intro}
Scaling architecture size has been the dominant driver of modern AI performance, with larger models consistently achieving better results \citep{kaplan2020scaling, hoffmann2022chinchilla,epoch2023aitrends}. However, AI development has reached the point where the scaling of computation becomes prohibitively expensive due to hardware and energy constraints \citep{strubell2019energy, thompson2020computational, sevilla2022compute}.  Adapting to the cost and hardware limitations of scaling, researchers have turned to sparse Mixture-of-Experts (s-MoE) architectures \citep{shazeer2017}, which are sparse realizations within the mixture-of-experts (MoE) paradigm codified by  \citet{jacobs1991adaptive}. 

In modern large-scale AI architectures, s-MoE layers --- consisting of several parallel subnetworks (``experts'') controlled by a ``sparse gate'' or {\it router} that selects data to route to them --- have largely replaced single submodules through which all data must pass. In these s-MoEs, for each input, the sparse-gating component selects a strict subset of experts (hence ``sparse'') to apply to that input. Thus, only a small subcomponent of an AI architecture is activated to process each piece of input data --- allowing models to have significantly more parameters while keeping inference and training costs manageable. As a testament to s-MoEs' utility, recent releases of OpenAI's GPT \citep{openai2024gpt4}, Google's Gemini \citep{google2024gemini}, and DeepSeek \citep{deepseek2024v3, deepseekai2026deepseekv4} have all leveraged s-MoE designs to improve efficiency and maintain performance scaling.

However, a crucial aspect of s-MoE design --- load balancing (controlling ``how many inputs per expert'') --- is mostly developed using trial-and-error motivated by heuristic insights (see Section \ref{sec:intro_balancing}). Learning to precisely and mathematically balance the load across experts, which reduces monetary losses from idle GPUs, could lead to enormous monetary savings for AI training.

\begin{figure}[h]
    \centering
    \includegraphics[width=\textwidth]{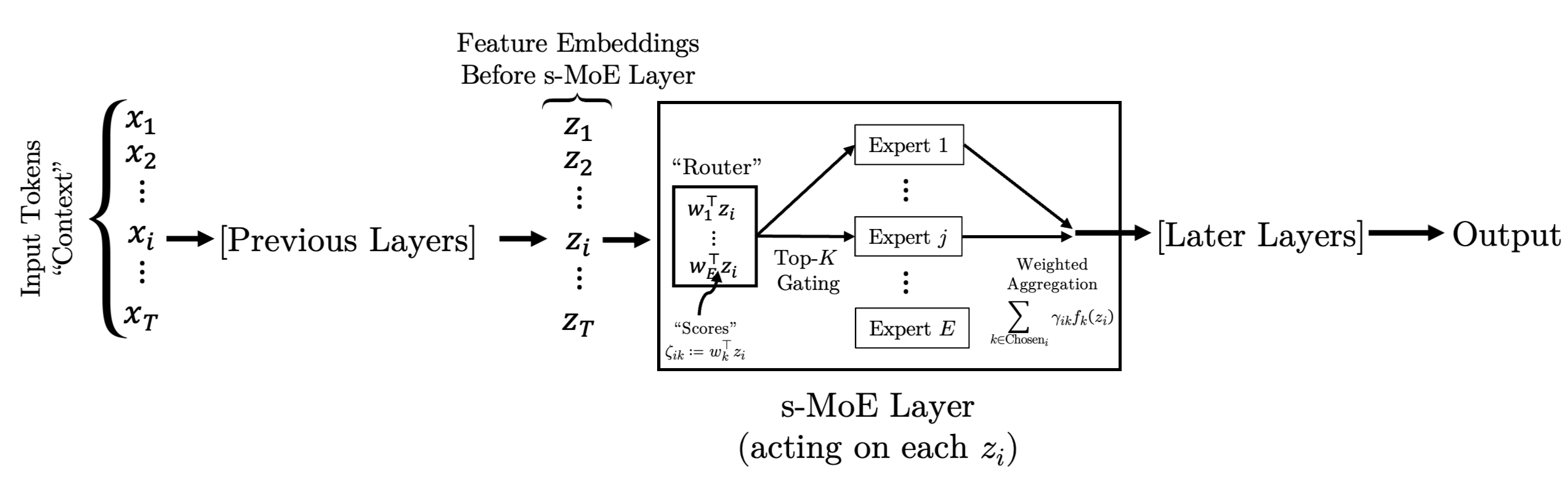}
    \caption{Schematic of a naïve s-MoE layer without load balancing.}
    \label{fig:moe_layer}
\end{figure}

\subsection{Naïve s-MoE Layers Without Load Balancing} \label{sec:setup}

Figure \ref{fig:moe_layer} describes the ``naïve'' setup for s-MoE layers within transformer-based AI models. In particular, the input is a series of {\it token embeddings} $x_1, x_2, ..., x_T$ where each $x_i$ is a high-dimensional vector corresponding (in language models) to a language unit such as ``Hel'', ``lo'', ``world'', etc. or (in vision models) a patch within an image. Each piece of input data (a sentence, an image patch, etc.) is decomposed into constituent tokens; each token is mapped to its vector embedding $x_i$; and those embeddings are input into the AI model. The entire tuple of vectors $\{x_i\}_{i=1}^T$ is called the {\it context} and $T$ is the {\it context length}.

Within the AI model, each of the original token embeddings $x_i$ is transformed into {\it feature embeddings} by each of the AI model's layers. In Figure \ref{fig:moe_layer}, to describe the action of some particular s-MoE layer, we use $\{z_i\}_{i=1}^T$ to denote the feature embeddings before that s-MoE layer.\sloppy

When a feature embedding ``enters'' an s-MoE layer with $E$ experts, we calculate an unnormalized affinity score $\zeta_{i,k}$ between $z_i$ and the $k$-th expert --- usually using an inner product:
$\zeta_{i,k} := w_k^\top z_i$. These scores are then normalized, typically using the ``softmax'' function, into the {\it affinity scores}:
\begin{equation}
\gamma_{i,k} := \softmax\left(\zeta_{i,k};\ \{\zeta_{i,k^\prime}\}_{k^\prime=1}^E\right) = \frac{\exp(\zeta_{i,k})}{\sum_{k^\prime =1}^E \exp(\zeta_{i,k^\prime})}\label{eq:affinity}
\end{equation}
The router then selects the Top-$K$ experts based on the $K$ largest $\gamma_{i,k}$. The final step in an s-MoE layer is to aggregate the outputs of the selected experts. This is done by computing a weighted sum of the selected experts' outputs:
\begin{equation}
    \sum_{k \in \text{ChosenExperts}_i} \gamma_{i,k} f_k(z_i),
    \label{eq:moe_aggregation}
\end{equation}
where $f_k$ represents the $k$-th expert. Note that the softmax is taken {\it before} the Top-$K$ selection, which is typically the preferred order in recent s-MoEs \citep{dai2024deepseekmoe, riquelme2021scaling}. Moreover, the softmax is monotonic, so it is equivalent to choose the Top-$K$ experts based on the $K$ largest $\{\gamma_{i,k}\}_{k=1}^E$ for each $i$, where $K < E$ for s-MoEs. This completes the description of the schematic in Figure \ref{fig:moe_layer}.
\vspace{6pt}

\subsection{Load Balancing of Experts: Background and Related Work}\label{sec:intro_balancing}
While conceptually simple, the naïve routing method of choosing the top-$K$ experts based only on $\{\gamma_{i,k}\}_{k=1}^E$ often causes load imbalance when assigning tokens to experts. This creates critical issues for both deployment and training. During deployment, imbalance results in underutilized GPUs hosting idling experts, which wastes costly computing resources. Although one could also balance GPU usage across {\it multiple requests} in high-traffic deployment scenarios (see, for example \citet{markovic2026robust} and citations therein), across-requests balancing operate in the deployment stage (not the training stage) and remain vulnerable to inefficiencies when traffic is low and expert affinities are skewed. Moreover, during training, such imbalance prevents effective learning across network parameters across all s-MoE layer experts: This is due to the uneven training of the experts that, in turn, induces a self-perpetuating cycle where the router preferentially selects better-trained experts while others become more underutilized and undertrained. Since every expert must be hosted on costly GPUs during training such imbalance could lead to significant monetary losses during training as well.

Several fixes have been proposed. The most commonly adopted approach is adding an auxiliary ``balancing loss'' directly to the training loss penalizing the network parameters during training for inducing imbalanced token allocations \citep{fedus2022switch, lepikhin2021gshard, shazeer2017}. However, as discussed in \citet[Section 2.2]{wang2024auxiliary}, this method interferes with the gradient updates of the performance-focused component of the objective. 

Another approach by \citet{lewis2021base} approximately solves --- via a truncated auction heuristic based on \citet{bertsekas1992auction} --- an integer program that balances the load across experts in every training iteration. However, generating an AI model's outputs for even one single batch of data (a ``forward pass'') requires significant computation time and memory since it requires calculating matrix multiplications and non-linear transformations defined by millions to billions of parameters. During training (as opposed to inference/deployment), there is an additional computational and memory overhead for computing and storing the backpropagated gradients (the ``backward pass''). Thus, it is inadvisable to spend additional time solving a multi-iterative subroutine (whether an auction algorithm or an integer program) for every s-MoE layer and every batch.

To address this problem, DeepSeek's auxiliary-loss-free (ALF-LB) \citep{wang2024auxiliary} procedure augments each expert with a bias $p_k$ using a {\it single-shot} update (as opposed to a multi-step subroutine), nudging tokens toward underloaded experts --- without interfering with training gradients as is the case when using auxiliary balancing losses\footnote{Specifically, the $p_k$ biases in ALF-LB are considered constants during the backpropagation phase of AI training.}. Notably, ALF-LB was used to successfully train the recent DeepSeekV3 \citep{deepseek2024v3} and DeepSeekV4 \citep{deepseekai2026deepseekv4} models.

\subsection{DeepSeek's ALF-LB Algorithm}\label{sec:alf_lb}
DeepSeek's ALF-LB procedure \citep{wang2024auxiliary} is as follows:
\begin{enumerate}
\item For each expert $k=1,\dots,E$, initialize a scalar shift parameter $p_{k}$ to be 0.
\item Perform a forward pass on a batch. During the forward pass, route token $i$ based on the experts with the highest shifted weights $\gamma_{ik}+p_{k}$.
\item Calculate the downstream network loss and update the main network parameters, treating the shifts $\{p_k\}$ as constants.
\item For each expert $k$, update its shift parameter as follows, where $u$ is a small constant (e.g., 0.001):
\begin{equation}
p_{k}\gets\begin{cases}
p_{k}-u & \text{if expert }k\text{ had load }>\ L;\\
p_{k}+u & \text{if expert }k\text{ had load }<\ L;\\
p_{k} & \text{otherwise}.
\end{cases}
\label{eq:alf_lb_update}
\end{equation}
\item Repeat steps 2-4 for each batch of input data. 
\end{enumerate}
In the original publication, \citet{wang2024auxiliary} chose $u=0.001$ and exhibited empirical benefits of this procedure on 1B to 3B parameter DeepSeekMoE models \citep{dai2024deepseekmoe}. Notably, this procedure was subsequently used to train the DeepSeekV3 \citep{deepseek2024v3} and DeepSeekV4 \citep{deepseekai2026deepseekv4} base models.

\subsection{Contributions and Organization of Paper}
Our main contribution is a rigorous theoretical framework for understanding and analyzing the ALF-LB procedure, with specific contributions detailed across different sections. First, in Section \ref{sec:load_balancing_PD}, we cast the ALF-LB procedure as a single-step primal-dual method for an assignment problem, connecting a state-of-the-art heuristic from large-scale AI to the operations research and primal-dual optimization literature for resource allocation such as those in \citet{bertsekas1992auction, bertsekas1998network,bertsekas2008auction}. However, the procedure we analyze differs from the aforementioned operations research problems since, as discussed in Section \ref{sec:intro_balancing}, the computational and memory requirements of performing a forward pass through an AI model do not allow for one to run multi-iterative procedures as subroutines with those forward passes. Instead, s-MoE balancing routines (such as ALF-LB) must be updated in a ``single-shot'' manner --- with computationally-minimal, constant-time updates per forward pass --- instead of relying on multi-iterative subroutines.

Then, in Section \ref{sec:deterministic_analysis}, we analyze this procedure in a stylized deterministic setting and establish several insightful structural properties: (i) a monotonic improvement condition for the Lagrangian objective (Theorem \ref{thm:Ldec}), (ii) a preference rule that moves tokens from overloaded to underloaded experts (Theorem \ref{thm:DSswitching}), and (iii) a band-stability guarantee showing that once expert loads enter an approximate-balance band, they remain there (Theorem \ref{thm:band_stability}). Finally, in Section \ref{sec:oco_analysis}, we extend our analysis to a more realistic online, stochastic setting by establishing a strong convexity property of the expected dual objective (Section \ref{sec:strong_convexity}) and using it to derive a logarithmic regret bound for the ALF-LB procedure (Theorem \ref{thm:logarithmic_regret}).

\subsubsection{Online Resource Allocation: Connections and Related Works}\label{sec:related_work}
It is insightful to compare this paper to another recent line of work at the intersection of AI implementation and operations research: the online resource allocation of multiple requests/queries in AI datacenters (see, for example, \citet{markovic2026robust} and citations therein) where computational requests arrive in an online, stochastic manner and must be optimally routed to servers in the datacenter to complete the job. In comparison, during a forward pass through a multi-layered s-MoE based AI architecture, any layer after an s-MoE layer must wait for all tokens to be routed, passed through assigned experts, and aggregated by the preceding s-MoE layer. For reference, DeepSeekV3 \citep{deepseek2024v3} and DeepSeekV4 \citep{deepseekai2026deepseekv4} both contain 61 s-MoE layers. Thus, unlike the routing of requests to datacenter servers, the allocation of tokens in s-MoE layers must be ``single-shot'' and computationally-minimal in order to not delay the sequential progression of the forward pass through the multi-layered AI architecture.

Another related line of works is \citet{balseiro2020dual, balseiro2021regularized, agrawal2014fast, jenatton2016online} (see \citet[Section 1.2]{balseiro2021regularized} and citations therein) that design and analyze primal-dual methods for solving online resource allocation problems by formulating them as online stochastic convex programs or regularized allocation problems. These prior works utilize dual descent and mirror descent techniques to manage global resource constraints which can be formulated as load balancing. However, the algorithms proposed in those works often require solving auxiliary optimization sub-routines such as linear programs, quadratic programs, or non-trivial projections during their updates. As discussed earlier, in the context of s-MoE training, multi-iterative subroutines are computationally impracticable because routing must occur in every s-MoE layer of the AI architecture during already-computationally-expensive forward passes, which does not allow for the extra overhead of solving auxiliary sub-routines at every s-MoE layer. Thus, in comparison, our paper instead analyzes a ``single-shot'' update framework, built specially to encompass DeepSeek's ALF-LB procedure \citep{wang2024auxiliary}, that updates the load balancing parameters with negligible effect on the speed of the forward pass.

\section{A Primal-Dual Framework for Optimal Load Balancing}\label{sec:load_balancing_PD}
Now, we establish a rigorous mathematical framework auxiliary-loss-free load balancing heuristics for s-MoE layers and, in particular, DeepSeek's ALF-LB method \citep{wang2024auxiliary}. In the remainder of the paper, for simplicity, we will refer to the normalized affinity scores $\gamma_{ik}$ (Equation \ref{eq:affinity}) as the ``affinity scores'' and adopt the convention of using them both for routing and aggregation.

\subsection{Allocation Problem: Integer Program and Relaxation}\label{sec:alloc_prob}

Consider the exact-balancing primal problem for assigning $T$ tokens to $E$ experts. As a starting point, we make the following assumptions and stylizations:
\begin{itemize}
    \item The number of tokens multiplied by the sparsity, $KT$, is exactly divisible by the number of experts $E$, so the perfectly balanced load is $L=KT/E$.
    \item The affinity scores $\gamma_{ik}$ are constant from iteration to iteration\footnote{This is a stylized assumption for the initial analysis in this section only. Later, in Section \ref{sec:oco_analysis}, we will consider the case where the affinity scores are new stochastic realizations from some distribution every iteration.}.
\end{itemize}

Hence, the target load is $L:=KT/E$ and perfect balance is characterized by the solution of the following integer program (IP):
\begin{equation}
\begin{aligned}
\max_{\left\{ x_{ik}\right\} }\  & \sum_{i,k}\gamma_{ik}x_{ik}\\
\text{s.t.} & \sum_{k}x_{ik}=K\quad\forall i=1,\dots,T\\
 & \sum_{i}x_{ik}=L\ \forall k=1,\dots,E\\
 & x_{ik}\in\left\{ 0,1\right\} \quad\forall i,k.
\end{aligned}
\label{eq:alloc_problem}
\end{equation}
In practice, it is typically inadvisable (in terms of both time and memory requirements) to solve an IP for every MoE layer and on each individual batch of data\footnote{One notable exception is the BASE layer heuristic invented by \citet{lewis2021base} which aims to approximately solve the IP using a truncated auction algorithm modeled after \citet{bertsekas1992auction}.}.

Instead, we first relax the IP to a linear program (LP) by replacing the integer constraint $x_{ik}\in\left\{ 0,1\right\}$ with $x_{ik} \in [0,1]$. It is routine to show that the IP and the LP relaxation have the same optimal value. The Lagrangian of the LP relaxation is
\begin{align}
\Lcal(x,y,p) 
&= \sum_{i,k}\gamma_{ik}x_{ik}+\sum_{i}y_{i}\left(K -\sum_{k}x_{ik}\right)+\sum_{k}p_{k}\left(\sum_{i}x_{ik}-L\right) \notag \\
&= \sum_{i,k}\left(\gamma_{ik}+p_{k}-y_{i}\right)x_{ik}+K \sum_{i}y_{i}-L\sum_{k}p_{k}.\label{eq:T1Lagrangian}
\end{align}
The corresponding dual problem is
\begin{align*}
\min_{\left\{ y_{i}\right\} ,\left\{ p_{k}\right\} }\  & K \sum_{i}y_{i}-L\sum_{k}p_{k}\\
\text{s.t. } & y_{i}-p_{k}\geq\gamma_{ik}\quad\forall i,k.
\end{align*}
However, even solving this LP relaxation to completion every iteration would still be too slow and often memory-infeasible.

\subsection{Deriving ALF-LB from Primal-Dual Principles}\label{sec:primal_dual_alflb}
We show that DeepSeek ALF-LB (Section \ref{sec:alf_lb}) can be formulated as a primal-dual procedure that performs a single-shot update per iteration for finding a critical point of the Lagrangian \eqref{eq:T1Lagrangian}. For conciseness, we introduce the following notation for the \textit{load} of the $k$-th expert at iteration $n$: 
\begin{equation}
A_{k}^{\left(n\right)}:=\sum_{i}x_{ik}^{\left(n\right)}.\label{eq:load}
\end{equation}
Indexing each training iteration with $n$, consider the following primal-dual scheme:
\begin{align}
\text{\textbf{Dual Update:} }p_{k}^{\left(n+1\right)} & \gets p_{k}^{\left(n\right)}+\epsilon_{k}^{\left(n\right)}\left(L-A_{k}^{\left(n\right)}\right) & \ \forall\ k.\label{eq:dualupdate}\\
\text{\textbf{Primal Update: }}x_{ik}^{\left(n+1\right)} & \gets\begin{cases}
1 & \text{if }k \in \topkind_{k^{\prime}}\big(\gamma_{ik^{\prime}}^{\left(n+1\right)}+p_{k^{\prime}}^{\left(n+1\right)}\big)\\
0 & \text{otherwise}
\end{cases} & \ \forall\ i,k,\label{eq:primalupdate}
\end{align}
where $\left\{ \epsilon_{k}^{\left(n\right)}\right\}$ are step-sizes and $\topkind_{k^\prime}(\cdot)$ gives the  indices that would induce the $K$-largest arguments. The Primal Update enforces $\sum_{k}x_{ik}=K$. Maximizing $\sum_{i,k}(\gamma_{ik}+p_{k}-y_{i})x_{ik}$ subject to this constraint is equivalent to choosing the top $K$ values of $\gamma_{ik}+p_k$ for each $i$, regardless of $y_i$. Thus we can simplify the Lagrangian by dropping the $y_i$ terms, which gives:
\begin{align}
\Lcal(x,p) & =\sum_{i,k}\left(\gamma_{ik}+p_{k}\right)x_{ik}-L\sum_{k}p_{k}.\label{eq:L}
\end{align}
Within this setup, the original DeepSeek ALF-LB update from \citet{wang2024auxiliary} described in \eqref{eq:alf_lb_update} corresponds to the step-size
\begin{align}
\epsilon_{k}^{\left(n\right)} & =\frac{u}{\left|L-A_{k}^{\left(n\right)}\right|}\ind\left\{\left|L-A_{k}^{(n)}\right|\neq 0\right\}, \quad \textbf{(DeepSeek ALF-LB Step-Size)}\label{eq:ds_step}
\end{align}
where $\ind\left\{...\right\}$ is the indicator function\footnote{For conciseness, we use the short-hand ``$\epsilon_{k}^{\left(n\right)} =\frac{u}{\left|L-A_{k}^{\left(n\right)}\right|}$'' to refer to this step-size in the remainder of this paper.}. Figure \ref{fig:convergence_behavior} illustrates the convergence behavior of this primal-dual scheme during the training of 1B-parameter DeepSeekMoE models \citep{dai2024deepseekmoe} with varying $\epsilon_{k}^{(n)}$ step-size choices. More experimental details are provided in Section \ref{sec:experiments}.

\begin{figure}[h]
    \centering
    \includegraphics[width=\textwidth]{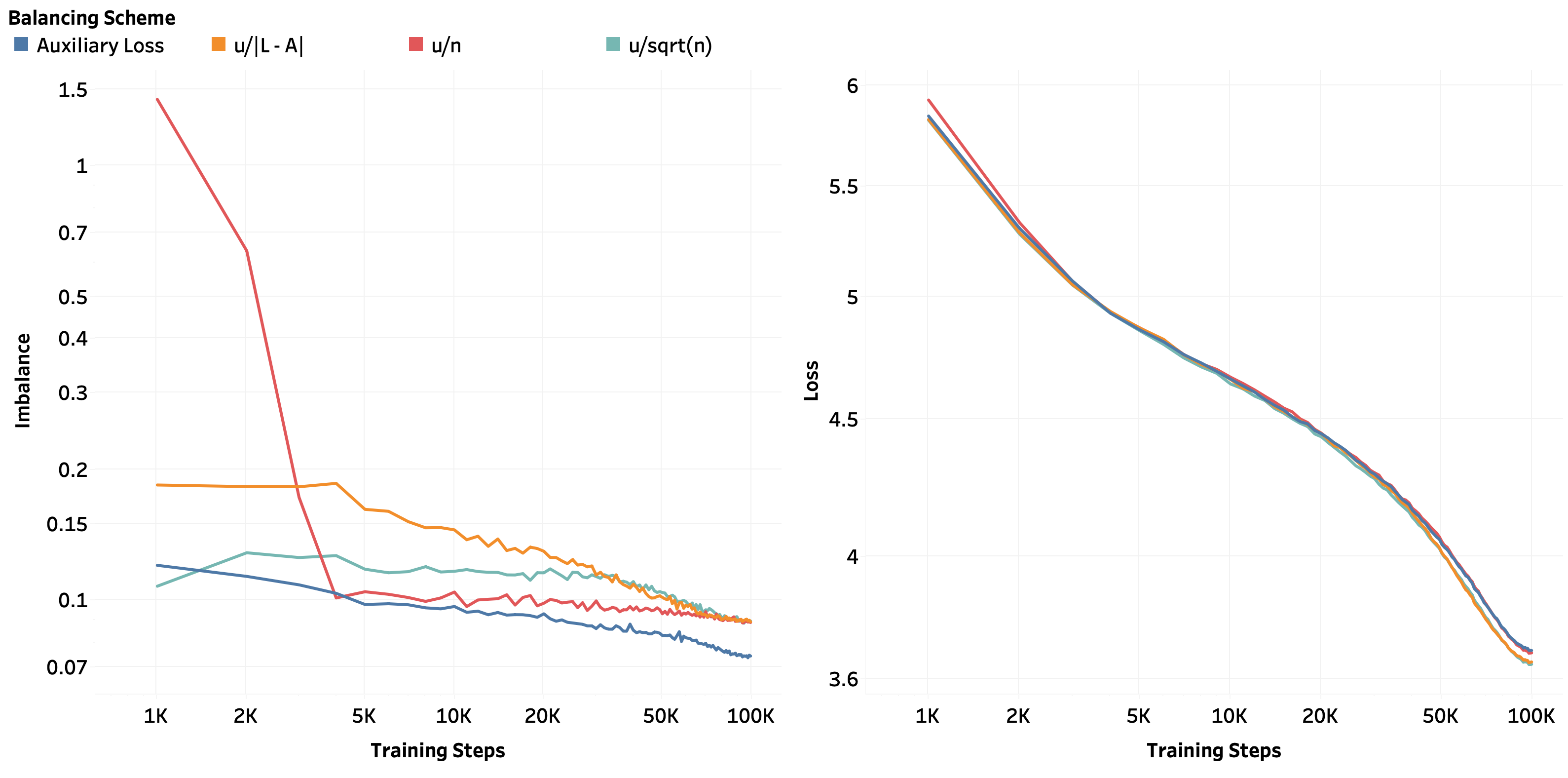}
    \caption{Validation set load imbalance and loss during the training of a 1B-parameter DeepSeekMoE model. Section \ref{sec:experiments} gives experiment details. {\it Left:} We measure the imbalance as the average load deviation from the target load $L=KT/E$ across all experts in the DeepSeekMoE-1B architecture. {\it Right:} We measure the loss on the validation set. }
    \label{fig:convergence_behavior}
\end{figure}

\section{Experimental Setup and Observations}\label{sec:experiments}

\subsection{Experimental Setup}

In all experiments in this paper (Figures \ref{fig:convergence_behavior}-\ref{fig:bias_histograms}), we train 1B-parameter DeepSeekMoE models \citep{dai2024deepseekmoe} for 100K steps on the next-token prediction task on the Salesforce WikiText-103 dataset \citep{merity2016pointer} with the cross-entropy loss. The text data is tokenized using the GPT-2 tokenizer \citep{radford2019language}.

Here, we will provide only a brief description of the DeepSeekMoE architecture for completeness and refer to \citet{dai2024deepseekmoe} for more in-depth details: The DeepSeekMoE architecture follows the paradigmatic practice of stacking decoder-only transformer layers \citep{vaswani2017attention} into a full large language model. In its simplest form, the transformer layer contains several sub-layers --- among them a multi-headed attention sub-layer and a multi-layer perceptron (MLP) sub-layer. For our setting of interest, modern s-MoE architectures \citep{shazeer2017, jiang2024mixtral,dai2024deepseekmoe} replace the MLP sub-layer of each transformer layer with an s-MoE sub-layer described in Sections \ref{sec:setup}-\ref{sec:intro_balancing}, where each parallel expert is typically a separate MLP. Additionally, the DeepSeekMoE architecture \citep{dai2024deepseekmoe} is specifically characterized by its use of ``granular segmentation'' (using narrower experts but increasing the total number of experts) and the inclusion of two ``shared experts'' that are always chosen by the gate\footnote{We will not include the shared experts within the theoretical framework presented in this paper because the shared experts represent a fixed computational load that does not require dynamic balancing. Additionally, omitting the shared experts from our theoretical formulation leads to cleaner and more concise analyses.}.

The architectural parameters of the 1B-parameter DeepSeekMoE models in our experiments are the same as those described in \citet[Table 5]{wang2024auxiliary}. For consistency with \citet{wang2024auxiliary}, we also use $E=64$ experts with sparsity level $K=6$. During training, we optimize all 1B parameters within the transformer backbone and prediction head of the DeepSeekMoE architectures starting from random initializations. Each model was trained on 8xH100/H200 GPUs with a batch size of 64 sequences/batch and 4096 tokens/sequence (so, $T\approx 262K$). To optimize the models, we use the AdamW \citep{loshchilov2018decoupled} optimizer. 

\paragraph{Balancing Schemes.} In our experiments, we compare three choices of the $k$-th expert step-sizes at iteration $n$ (denoted $\epsilon^{(n)}_k$, see Section \ref{sec:primal_dual_alflb}) in the ALF-LB balancing scheme framework. In particular, given some balancing hyperparameter $u$, we compare the following schemes:
\begin{itemize}
\item $\epsilon^{(n)}_k = \frac{u}{\left|L-A^{(n)}_k\right|}$ (Original DeepSeek ALF-LB from \citet{wang2024auxiliary})
\item $\epsilon^{(n)}_k = \frac{u}{n}$
\item $\epsilon^{(n)}_k = \frac{u}{\sqrt{n}}$
\end{itemize}
\noindent Additionally, we include a comparison with a fourth scheme that trains with an auxiliary loss \citep{shazeer2017, lepikhin2021gshard, fedus2022switch,wang2024auxiliary}. We calculate the auxiliary loss with the method described in \citet[Section 2.2]{wang2024auxiliary}. The auxiliary loss is multiplied by a ``trade-off parameter'' that we will, for consistency, also denote by $u$ and then added to the main cross-entropy loss.

\paragraph{Hyperparameter Search.} For each of the four scheduling schemes, we conducted hyperparameter search over the following hyperparameters:
\begin{itemize}
\item balancing constants $u\in\{1e{-}4, 1e{-}3, 1e{-}2, 1e{-}1, 1, 10\}$, 
\item learning rates $\lr\in\{1e{-}5, 1e{-}4, 1e{-}3\}$, and 
\item weight decay $\wdy\in\{0.01, 0.1, 0.001\}$. 
\end{itemize}
Thus, we trained $4\times 6 \times 3 \times 3 = 216$ separate 1B-parameter DeepSeekMoE models to conduct this search. Then, for each of the four scheduling schemes, we select the hyperparameter setting that achieves the best cross-entropy loss on a held-out validation set to be shown in the experimental plots and tables in this paper. We found that 
\begin{itemize}
\item $\lr = 1e{-}4$ and $\wdy=1e{-}1$ consistently led to the best validation loss across all settings;
\item the $u/n$ and auxiliary loss scheduling schemes performed the best with parameter $u=1$; and
\item the $u/|L-A^{(n)}_k|$ and $u/\sqrt{n}$ scheduling schemes performed the best with $u=1e{-}3$.
\end{itemize}

\subsection{Experimental Observations}
We make some interesting empirical observations from our experiments that are of separate interest from the theoretical framework proposed in this paper.

Firstly, we found that, for the original ``constant update'' scheme considered by \cite{wang2024auxiliary} (which corresponds to $u/|L-A^{(n)}_k|$ in our formalization), our hyperparameter search also yielded $u=1e{-}3$ to be the optimal balancing constant, which corroborates the same observation from \citet{wang2024auxiliary}.

Secondly, Table \ref{tab:final_compare} reports the final validation loss and overall imbalance of the different balancing schemes at the end of training. Observe that the $u/\sqrt{n}$ scheme achieves the lowest validation loss (best predictive performance) but the highest imbalance (worst computational efficiency); in contrast, the auxiliary loss approach \citep{shazeer2017, lepikhin2021gshard, fedus2022switch} achieves the lowest imbalance (best computational efficiency) but the highest validation loss (worst predictive performance). The $u/n$ scheme (which we will analyze in Section \ref{sec:oco_analysis} through the lens of online optimization) and \citet{wang2024auxiliary}'s original $u/|L-A^{(n)}_k|$ scheme achieve a balance between validation loss and imbalance --- with $u/n$ achieving slightly better balance and $u/|L-A^{(n)}_k|$ achieving slightly better predictive performance.

\begin{table}[ht]
  \centering
  \caption{Comparison of cross-entropy loss on validation data and overall imbalance at the end of training for different scheduling schemes. Experiment details in Section \ref{sec:experiments}.}
  \label{tab:final_compare}
  \begin{tabular}{lrr}
  \hline
  \textbf{Balancing Scheme} & \textbf{Validation Loss} & \textbf{Overall Imbalance} \\
  \hline
  Auxiliary Loss & 3.68999 & \textbf{0.07443} \\
  DeepSeek's Original ALF-LB $\left(u/|L-A^{(n)}_k|\right)$      & 3.65369 & 0.08928 \\
  $u/n$          & 3.68228 & 0.08893 \\
  $u/\sqrt{n}$   & \textbf{3.64642} & 0.08961 \\
  \hline
  \end{tabular}
\end{table}

\section{Convergence Analysis for the Deterministic Case}\label{sec:deterministic_analysis}

\subsection{Section Assumptions}\label{sec:deterministic_assumptions}
To start, we derive theoretical guarantees for the convergence of the procedure described by \eqref{eq:dualupdate} and \eqref{eq:primalupdate} in a stylistic setting where we {\bf assume here in Section \ref{sec:deterministic_analysis} only} that scores $\gamma_{ik}$ are fixed and deterministic. We will later consider the case where the scores are stochastic in Section \ref{sec:oco_analysis}. Additionally, since the scores $\gamma_{ik}$ are the output of softmax transformations (Section \ref{sec:setup}), we assume $\gamma_{ik}\in(0,1)$ for all $i,k$.

\subsection{Monotonicity of the Lagrangian}
Towards showing the convergence of this procedure, we will show a monotonic improvement condition for the Lagrangian.

\subsubsection{Lagrangian Characterization}

We start by characterizing the change in the Lagrangian in each iteration. Two helpful abstractions are the \textit{assignment set}
\begin{equation}\label{eq:assignment_set}
  \alpha_n(i) := \topkind_{k^{\prime}}\!\left(\gamma_{ik^{\prime}}+p_{k^{\prime}}^{\left(n\right)}\right),
\end{equation}
which gives the set of indices of the Top-$K$ experts assigned to token $i$ at iteration $n$; and the \textit{switching benefit} of token $i$ 
\begin{equation}
  b^{\left(n+1\right)}(i)
  := \sum_{k\in\alpha_{n+1}(i)} \left(\gamma_{ik}+p_{k}^{\left(n+1\right)}\right)
   - \sum_{k\in\alpha_{n}(i)} \left(\gamma_{ik}+p_{k}^{\left(n+1\right)}\right),
\label{eq:benefit}
\end{equation}
which captures the gain in the Lagrangian-affinity term when token $i$ replaces its Top-$K$ set $\alpha_n(i)$ by $\alpha_{n+1}(i)$ with respect to the fixed dual variables at iteration $n{+}1$. Since $\alpha_{n+1}(i)$ is chosen as a Top-$K$ set with respect to $\gamma_{ik}+p_k^{(n+1)}$, we have $b^{\left(n+1\right)}(i) \ge 0$ for all $i$.

\begin{thm}\label{thm:Ldec}\textbf{(Change in Lagrangian)} 
Under the assumptions in Section \ref{sec:deterministic_assumptions} and using the procedure described in Steps \eqref{eq:dualupdate}-\eqref{eq:primalupdate}, the following holds for the Lagrangian \eqref{eq:L}:
\[
\Lcal\left(x^{\left(n+1\right)},p^{\left(n+1\right)}\right)-\Lcal\left(x^{\left(n\right)},p^{\left(n\right)}\right)=\sum_{i}b^{\left(n+1\right)}\left(i\right)-\sum_{k}\epsilon_{k}^{\left(n\right)}\left(A_{k}^{\left(n\right)}-L\right)^{2}.
\]
\end{thm}
\begin{proof}
Using the Lagrangian \eqref{eq:L} definition, we have:
\begin{multline*}
\Lcal\left(x^{\left(n+1\right)},p^{\left(n+1\right)}\right)-\Lcal\left(x^{\left(n\right)},p^{\left(n\right)}\right) \\ = \sum_{i,k}\left(\gamma_{ik}+p_{k}^{\left(n+1\right)}\right)x_{ik}^{\left(n+1\right)} - \sum_{i,k}\left(\gamma_{ik}+p_{k}^{\left(n\right)}\right)x_{ik}^{\left(n\right)}
- L\sum_k (p_k^{(n+1)} - p_k^{(n)}) \\
= \sum_i \left[\sum_{k\in\alpha_{n+1}(i)}\left(\gamma_{ik}+p_{k}^{(n+1)}\right) - \sum_{k\in\alpha_n(i)}\left(\gamma_{ik}+p_{k}^{(n)}\right)\right] - L\sum_k \left(p_k^{(n+1)} - p_k^{(n)}\right) \\
\text{by assignment set definition \eqref{eq:assignment_set}}
\\
= \sum_i \left[ b^{(n+1)}(i) + \sum_{k\in\alpha_n(i)}\left(\gamma_{ik}+p_{k}^{(n+1)}\right) - \sum_{k\in\alpha_n(i)}\left(\gamma_{ik}+p_{k}^{(n)}\right) \right] - L\sum_k \left(p_k^{(n+1)} - p_k^{(n)}\right) \\
\text{by switching benefit definition \eqref{eq:benefit}}
\\
= \sum_i b^{(n+1)}(i) + \sum_i \sum_{k\in\alpha_n(i)}\left(p_k^{(n+1)} - p_k^{(n)}\right) - L\sum_k \left(p_k^{(n+1)} - p_k^{(n)}\right) \\
= \sum_i b^{(n+1)}(i) + \sum_k A_k^{(n)}\left(p_k^{(n+1)} - p_k^{(n)}\right) - L\sum_k \left(p_k^{(n+1)} - p_k^{(n)}\right) \\
\text{by expert load definition \eqref{eq:load}}
\\
= \sum_i b^{(n+1)}(i) + \sum_k \left(A_k^{(n)}-L\right)\left(p_k^{(n+1)} - p_k^{(n)}\right) \\
= \sum_i b^{(n+1)}(i) - \sum_k \epsilon_k^{(n)}\left(A_k^{(n)}-L\right)^2\\
\text{by step \eqref{eq:dualupdate} definition}.
\end{multline*}
This completes the proof.
\end{proof}

Thus, Theorem \ref{thm:Ldec} shows that the improvement in the Lagrangian is the difference between the total switching benefit and the squared sum of load imbalances weighted by step-sizes). 

In fact, we can further characterize the switching benefit in terms of the expert choice changes between iterations. Specifically, define the sets of removed and newly selected experts, respectively, for some token~$i$ between iterations $n+1$ and $n$:
\[
\rmvinds^{(n+1)}(i):=\alpha_n(i)\setminus\alpha_{n+1}(i),\qquad
\addinds^{(n+1)}(i):=\alpha_{n+1}(i)\setminus\alpha_n(i).
\]
We will denote the corresponding count of changed assignments as
\[
|\newinds^{(n+1)}(i)|:=|\rmvinds^{(n+1)}(i)|=|\addinds^{(n+1)}(i)|,
\]
where the index sets are necessarily equal-sized since we always select the top-$K$ experts. As such, we can define an arbitrary bijection $\bijection_i^{(n)}$ between $\addinds^{(n+1)}(i)$ and $\rmvinds^{(n+1)}(i)$ and define a set of entering-exiting pairs relative to $\bijection_i^{(n)}$:
\begin{equation*}
\allpairs_i^{(n)} := \left\{(k^+,k^-) \in \addinds^{(n+1)}(i) \times \rmvinds^{(n+1)}(i) \mid k^+ = \bijection_i^{(n)}(k^-)\right\}.
\end{equation*}
We can then characterize the switching benefit using $\allpairs_i^{(n)}$.
\begin{prop}\textbf{(Switching Benefit Decomposition)}\label{prop:pairwise_switching}
Assume the setting in Section \ref{sec:deterministic_assumptions} using the procedure described in Steps \eqref{eq:dualupdate}-\eqref{eq:primalupdate} as well as that there are no ties in bias-shifted scores $\gamma_{ik}+p_k^{(n)}$. Then, for a fixed token $i$ at iteration $n$, we can write the switching benefit as
\begin{equation}\label{eq:pairwise_benefit}
  b_i^{(n+1)}
  = \sum_{(k^+,k^-) \in \allpairs_i^{(n)}}
      \left[\left(\gamma_{ik^+}+p_{k^+}^{(n+1)}\right)
           - \left(\gamma_{ik^-}+p_{k^-}^{(n+1)}\right)\right].
\end{equation}
\end{prop}
\begin{proof}
By definition of $b_i^{(n+1)}$ (Equation \eqref{eq:benefit}),
\begin{align*}
  b_i^{(n+1)}
  & = \sum_{k\in\alpha_{n+1}(i)}\left(\gamma_{ik}+p_k^{(n+1)}\right)
    - \sum_{k\in\alpha_n(i)}\left(\gamma_{ik}+p_k^{(n+1)}\right)\\
  & = \sum_{k\in \addinds^{(n+1)}(i)}\left(\gamma_{ik}+p_k^{(n+1)}\right)
    - \sum_{k\in \rmvinds^{(n+1)}(i)}\left(\gamma_{ik}+p_k^{(n+1)}\right)\\
  & = \sum_{(k^+,k^-) \in \allpairs_i^{(n)}}
      \left[\left(\gamma_{i{k^+}}+p_{k^+}^{(n+1)}\right)
       - \left(\gamma_{ik^-}+p_{k^-}^{(n+1)}\right)\right].
\end{align*}
This proves the desired result.
\end{proof}

While Theorems \ref{thm:Ldec} and Proposition \ref{prop:pairwise_switching} apply to arbitrary step-size $\epsilon_k^{(n)}$ choices, additional insights can be derived when specializing to the \cite{wang2024auxiliary}'s original step-size choice.

\subsection{Analysis of DeepSeek's Original Step-Size Choice}

Under \cite{wang2024auxiliary}'s original step-size choice (Equations \eqref{eq:alf_lb_update} and \eqref{eq:ds_step}), we derive more precise behaviors for ALF-LB in our theoretical setting.

To start, we can infer the following bounds on differences of bias-shifted score for entering-exiting pairs that will be useful in the subsequent analysis. Note that we will use the $\sign(\cdot): \mathbb{R} \to \{-1,0,1\}$ function with the convention $\sign(0)=0$.

\begin{lem}\textbf{(Pairwise Bounds)}\label{lem:pairwise_bounds}
  Assume the setting in Section \ref{sec:deterministic_assumptions} using the procedure described in Steps \eqref{eq:dualupdate}-\eqref{eq:primalupdate} with DeepSeek's original step-size \eqref{eq:ds_step} as well as that there are no ties in bias-shifted scores $\gamma_{ik}+p_k^{(n)}$. Then, for a fixed token $i$, the following hold between iterations $n$ and $n+1$:
\begin{enumerate}[label=\alph*.]
\item Every pair $(k^+,k^-)\in \addinds^{(n+1)}(i)\times \rmvinds^{(n+1)}(i)$ satisfies
\begin{align}
  0 &< \left(\gamma_{ik^+}+p_{k^+}^{(n+1)}\right) - \left(\gamma_{ik^-}+p_{k^-}^{(n+1)}\right)
     < u\left(\sign\!\left(L-A_{k^+}^{(n)}\right)-\sign\!\left(L-A_{k^-}^{(n)}\right)\right)\label{eq:pairwise_bound1}\\
     &< 2u,\label{eq:pairwise_bound2}
\end{align}
and
\begin{equation}\label{eq:neg_pairwise_bound}
  -2u<\gamma_{ik^+}+p_{k^+}^{(n)}-\left(\gamma_{ik^-}+p_{k^-}^{(n)}\right)<0.
\end{equation}
\item If $\alpha_{n+1}(i)\neq \alpha_n(i)$, then the switching benefit of token $i$ is bounded by the number of changed assignments:
\[
  0<b_{i}^{\left(n+1\right)}<2u\,|\newinds^{(n+1)}(i)|.
\]
If $\alpha_{n+1}(i) = \alpha_n(i)$, then $b_i^{(n+1)} = 0$.
\end{enumerate}
\end{lem}
\begin{proof}
Fix any pair $(k^+,k^-)\in \addinds^{(n+1)}(i)\times \rmvinds^{(n+1)}(i)$. By definition, $k^-\in\alpha_n(i)$ and $k^+\notin\alpha_n(i)$, so
\begin{equation}\label{eq:iter_n_choice}
  \gamma_{ik^+}+p_{k^+}^{(n)} - \gamma_{ik^-}+p_{k^-}^{(n)} < 0.
\end{equation}
Similarly, $k^+\in\alpha_{n+1}(i)$ and $k^-\notin\alpha_{n+1}(i)$, so
\begin{equation}\label{eq:iter_next_choice}
  0
  < \left(\gamma_{ik^+}+p_{k^+}^{(n+1)}\right) - \left(\gamma_{ik^-}+p_{k^-}^{(n+1)}\right).
\end{equation}
Denote $s_k^{(n)} := \sign\!\left(L-A_k^{(n)}\right)$. Using the update \eqref{eq:alf_lb_update}, we can rewrite
\begin{equation}\label{eq:iter_nnext_choice}
  \left(\gamma_{ik^+}+p_{k^+}^{(n+1)}\right) - \left(\gamma_{ik^-}+p_{k^-}^{(n+1)}\right)
  = \left(\gamma_{ik^+}+p_{k^+}^{(n)} - \gamma_{ik^-}-p_{k^-}^{(n)}\right)
    + u\left(s_{k^+}^{(n)}-s_{k^-}^{(n)}\right).
\end{equation}
Combining \eqref{eq:iter_nnext_choice} with \eqref{eq:iter_n_choice} gives \eqref{eq:pairwise_bound1}. 

The inequality \eqref{eq:pairwise_bound2} follows since $s_{k^\pm}^{(n)}\in\{-1,0,1\}$, so sign differences are at most 2.

Next, rearranging \eqref{eq:iter_nnext_choice} gives
\[
  \gamma_{ik^+}+p_{k^+}^{(n)}-\left(\gamma_{ik^-}+p_{k^-}^{(n)}\right)
  = \left(\gamma_{ik^+}+p_{k^+}^{(n+1)}\right) - \left(\gamma_{ik^-}+p_{k^-}^{(n+1)}\right)
    - u\left(s_{k^+}^{(n)}-s_{k^-}^{(n)}\right),
\]
and using \eqref{eq:iter_next_choice} together and again bounding sign differences by 2 yields \eqref{eq:neg_pairwise_bound}.

Finally, if $\alpha_{n+1}(i)\neq \alpha_n(i)$, then $|\newinds^{(n+1)}(i)|=|\allpairs_i^{(n)}|>0$. By Proposition~\ref{prop:pairwise_switching}, $b_i^{(n+1)}$ is a sum of $|\newinds^{(n+1)}(i)|$ differences of the form
\[
  \left(\gamma_{ik^+}+p_{k^+}^{(n+1)}\right) - \left(\gamma_{ik^-}+p_{k^-}^{(n+1)}\right),
\]
and by Item (a) each such term lies in $(0,2u)$. Summing gives $0<b_i^{(n+1)}<2u\,|\newinds^{(n+1)}(i)|$. 
If $\alpha_{n+1}(i) = \alpha_n(i)$, then $b_i^{(n+1)} = 0$ trivially by definition. This proves Item (b).
\end{proof}

One immediate consequence of Proposition \ref{prop:pairwise_switching} and Lemma \ref{lem:pairwise_bounds} is a more precise characterization of the Lagrangian change under DeepSeek's original step-size.

\begin{thm} \label{thm:LdecBound}
  \textbf{(Lagrangian with DeepSeek Step-Size)} Consider the setting in Section \ref{sec:deterministic_assumptions} using the procedure described in Steps \eqref{eq:dualupdate}-\eqref{eq:primalupdate} with step-size DeepSeek's original step-size \eqref{eq:ds_step} as well as that there are no ties in bias-shifted scores $\gamma_{ik}+p_k^{(n)}$. Then, the Lagrangian change in Theorem \ref{thm:Ldec} simplifies to
\begin{equation}\label{eq:Ldec_deepseek}
\Lcal\left(x^{\left(n+1\right)},p^{\left(n+1\right)}\right)-\Lcal\left(x^{\left(n\right)},p^{\left(n\right)}\right)=\sum_{i}b^{\left(n+1\right)}\left(i\right)-u\sum_{k}\left|A_{k}^{\left(n\right)}-L\right|.
\end{equation}
Furthermore, let $\Scal^{\left(n+1\right)}$ denote the index set of tokens whose assignment sets changed between iterations $(n+1)$ and $n$.
Then,
\begin{equation}\label{eq:Ldec_deepseek_bound}
\Lcal\left(x^{\left(n+1\right)},p^{\left(n+1\right)}\right)-\Lcal\left(x^{\left(n\right)},p^{\left(n\right)}\right)
\;\le\; u \left[2\sum_{i\in\Scal^{(n+1)}} |\newinds^{(n+1)}(i)|-\sum_{k}\left|A_{k}^{\left(n\right)}-L\right|\right],
\end{equation}
with strict inequality whenever $\Scal^{(n+1)}\neq\emptyset$.
\end{thm}
\begin{proof}
Substituting step-size \eqref{eq:ds_step} into Theorem~\ref{thm:Ldec} yields \eqref{eq:Ldec_deepseek}.

Next, Lemma~\ref{lem:pairwise_bounds}b states that $0<b_i^{(n+1)}<2u\,|\newinds^{(n+1)}(i)|$ whenever $i\in\Scal^{(n+1)}$ and $b_i^{(n+1)}=0$ otherwise. Summing over $i$ gives
\[
  \sum_i b_i^{(n+1)} \le 2u \sum_{i\in\Scal^{(n+1)}} |\newinds^{(n+1)}(i)|,
\]
with strict inequality if $\Scal^{(n+1)}\neq\emptyset$. Substituting this upper bound into \eqref{eq:Ldec_deepseek} yields \eqref{eq:Ldec_deepseek_bound}.
\end{proof}

Additionally, we can derive interesting implications for when a token's assignment set changes. 
\begin{thm}\label{thm:DSswitching}
  \textbf{(Assignment Change Implications)} Assume the setting of Theorem \ref{thm:LdecBound}. Then, for any entering-exiting index pair $(k^+,k^-)\in \addinds^{(n+1)}(i)\times \rmvinds^{(n+1)}(i)$,
\[
\sign\!\left(L-A_{k^+}^{(n)}\right) > \sign\!\left(L-A_{k^-}^{(n)}\right).
\]
In other words, entering experts are strictly lower in the ordering 
\[\textrm{Overloaded}\succ\textrm{Balanced}\succ\textrm{Underloaded}
\]
than exiting experts.
\end{thm}
\begin{proof}
By Lemma~\ref{lem:pairwise_bounds}a, we have
\[
  0
  < \left(\gamma_{ik^+}+p_{k^+}^{(n+1)}\right) - \left(\gamma_{ik^-}+p_{k^-}^{(n+1)}\right)
  < u\left(\sign\!\left(L-A_{k^+}^{(n)}\right)-\sign\!\left(L-A_{k^-}^{(n)}\right)\right).
\]
The result follows after recalling $u > 0$.
\end{proof}

Next, we show that, if all experts' qualitative state (overloaded vs. balanced vs. underloaded) do not change between iterations, the Lagrangian cannot increase and, in fact, strictly decreases whenever there is any imbalanced experts.
\begin{thm}\textbf{(Lagrangian Monotonicity)}
\label{thm:LdecK} 
Assume the setting of Theorem \ref{thm:LdecBound}. Consider the index sets of overloaded, balanced, and underloaded experts, respectively, at iteration $n$:
\[
  \ovrload^{(n)} := \{k : A_k^{(n)} > L\},\qquad
  \balload^{(n)} := \{k : A_k^{(n)} = L\},\qquad
  \udrload^{(n)} := \{k : A_k^{(n)} < L\}.
\]
Assume the imbalance states stay unchanged between iterations $n$ and $n+1$:
\[
  \ovrload^{(n+1)} = \ovrload^{(n)}\quad\text{and}\quad \udrload^{(n+1)} = \udrload^{(n)}.
\]
(So, necessarily, $\balload^{(n+1)} = \balload^{(n)}$ as well.) Then,
\[
\Lcal\left(x^{\left(n+1\right)},p^{\left(n+1\right)}\right)-\Lcal\left(x^{\left(n\right)},p^{\left(n\right)}\right)\le 0,
\]
with strict inequality whenever there exists any imbalanced expert.
\end{thm}
\begin{proof}
For conciseness, let $s_k^{(n)} := \sign\!\left(L-A_k^{(n)}\right)$. Combining Proposition~\ref{prop:pairwise_switching} and Lemma~\ref{lem:pairwise_bounds}a yields
\[
  b_i^{(n+1)}
  \leq \sum_{(k^+_i,k^-_i) \in \allpairs_i^{(n)}}
  u\left(s_{{k_i}^+}^{(n)}-s_{{k_i}^-}^{(n)}\right).
\]
Theorem \ref{thm:DSswitching} implies $\left(s_{{k_i}^+}^{(n)}-s_{{k_i}^-}^{(n)}\right) > 0$, so equality holds if and only if $\allpairs_i^{(n)}=\emptyset$ i.e. when there are no changes in the assignment set for token $i$ between iterations $n$ and $n+1$.

Summing over all tokens $i$ yields
\begin{equation}\label{eq:b_sum}
  \sum_i b^{(n+1)}(i)
  \;\le\;
  u\sum_i\sum_{(k_i^+,k_i^-) \in \allpairs_i^{(n)}} \left(s_{{k_i}^+}^{(n)}-s_{{k_i}^-}^{(n)}\right),
\end{equation}
where equality holds if and only if $\allpairs_i^{(n)}=\emptyset$ for all $i$.

Consider any individual expert $k$. Every time expert $k$ appears as an entering expert ($k=k_i^+$) for some token $i$, it contributes $+s_{k}^{(n)}$ to the RHS of \eqref{eq:b_sum}. Correspondingly, every time expert $k$ appears as an exiting expert ($k=k_i^-$), it contributes $-s_k^{(n)}$ to the RHS of \eqref{eq:b_sum}. Thus, the total contribution of each individual expert $k$ to the RHS of \eqref{eq:b_sum} is
\[
  s_k^{(n)}\sum_i\left(x_{ik}^{(n+1)}-x_{ik}^{(n)}\right)
  = s_k^{(n)}\left(A_k^{(n+1)}-A_k^{(n)}\right),
\]
where the equality follows from the definition \eqref{eq:load}. Summing over $k$ allows us to rewrite \eqref{eq:b_sum} as
\begin{equation}\label{eq:b_sum2}
  \sum_i b^{(n+1)}(i)
  \;\le\;
  u\sum_{k=1}^E s_k^{(n)}\left(A_k^{(n+1)}-A_k^{(n)}\right).
\end{equation}
Observe,
\begin{itemize}
  \item Any balanced expert $k \in \balload^{(n)} = \balload^{(n+1)}$ satisfies $A_k^{(n+1)}=A_k^{(n)}=L$. So, they contribute $0$ to the RHS sum in \eqref{eq:b_sum2}.
  \item Any overloaded expert $k \in \ovrload^{(n)} = \ovrload^{(n+1)}$ has $s_k^{(n)}=-1$ by definition. Theorem \ref{thm:DSswitching} implies that changes in any token's assignment set can not have an overloaded expert as the entering expert, so $A_k^{(n+1)} \le A_k^{(n)}$ for that overloaded expert.
  \item Any underloaded expert $k \in \udrload^{(n)} = \udrload^{(n+1)}$ has $s_k^{(n)}=1$ by definition. Similarly, Theorem \ref{thm:DSswitching} implies $A_k^{(n+1)} \ge A_k^{(n)}$ for that underloaded expert.
  \item Every token is assigned to $K$ experts, so the total load is conserved: $\sum_k A_k^{(n+1)}=\sum_k A_k^{(n)} = KT$. Thus, the total gain of underloaded experts equals the total loss of overloaded experts:
  \[
    \sum_{k\in \udrload^{(n)}} \left(A_k^{(n+1)}-A_k^{(n)}\right)
    = \sum_{k\in \ovrload^{(n)}} \left(A_k^{(n)} - A_k^{(n+1)}\right).
  \]
\end{itemize}
Thus, we can rewrite the sum in \eqref{eq:b_sum2} in terms of just the overloaded experts:
\begin{equation}\label{eq:only_overload}
  \sum_{k=1}^E s_k^{(n)}\left(A_k^{(n+1)}-A_k^{(n)}\right)
  = 2\sum_{k\in \ovrload^{(n)}}\left(A_k^{(n)}-A_k^{(n+1)}\right).
\end{equation}
For $k\in \ovrload^{(n)} = \ovrload^{(n+1)}$, we have $A_k^{(n+1)}\ge L+1$ by definition. So,
\[
  A_k^{(n)}-A_k^{(n+1)}
  \le A_k^{(n)}-(L+1)
  = \left(A_k^{(n)}-L\right)-1.
\]
Summing over $k\in \ovrload^{(n)}$ gives
\begin{equation}\label{eq:only_overload2}
  \sum_{k\in \ovrload^{(n)}}\left(A_k^{(n)}-A_k^{(n+1)}\right)
  \le \sum_{k\in \ovrload^{(n)}}\left(A_k^{(n)}-L\right) - |\ovrload^{(n)}|.
\end{equation}
Moreover, since total load is conserved and $L = KT/E$ by definition,
\begin{align*}
  0 &= \sum_k \left(A_k^{(n)} - L\right) \\
  &= \sum_{k\in \ovrload^{(n)}} \underbrace{\left(A_k^{(n)} - L\right)}_{>0} + \sum_{k\in \udrload^{(n)}} \underbrace{\left(A_k^{(n)} - L\right)}_{<0} + \sum_{k\in \balload^{(n)}} \underbrace{\left(A_k^{(n)} - L\right)}_{=0}\\
  &= \sum_{k\in \ovrload^{(n)}} \left|A_k^{(n)} - L\right| - \sum_{k\in \udrload^{(n)}} \left|A_k^{(n)} - L\right|.
\end{align*}
From this, we can deduce
\begin{equation}\label{eq:half_abs}
  \sum_{k\in \ovrload^{(n)}}\left(A_k^{(n)}-L\right)
  = \tfrac12\sum_k\left|A_k^{(n)}-L\right|.
\end{equation}
Combining \eqref{eq:b_sum2}, \eqref{eq:only_overload}, \eqref{eq:only_overload2}, and \eqref{eq:half_abs} gives
\[
  \sum_i b^{(n+1)}(i)
  \;\le\;
  u\left(\sum_k\left|A_k^{(n)}-L\right| - 2|\ovrload^{(n)}|\right).
\]
Substituting this bound into Theorem~\ref{thm:LdecBound}, yields
\[
  \Lcal\left(x^{(n+1)},p^{(n+1)}\right)-\Lcal\left(x^{(n)},p^{(n)}\right)
  \le -2u\,|\ovrload^{(n)}|.
\]
Observe the the LHS is strictly negative whenever there exists any imbalanced experts.
\end{proof}

\subsubsection{Gap Analysis and Maximum Token Movement}

Within the same iteration, we can consider the gap between the scores or biases of two different experts for a given token. In particular,for some token $i$ and experts $k$ and $k'$, define {\it score gap} as 
\[
\gap\gamma_{k^{\prime}k}^{i} := \gamma_{ik^{\prime}}-\gamma_{ik}.
\]
Similarly, for two experts $k$ and $k'$, define the {\it bias gap} as 
\begin{equation}\label{eq:bias_gap}
\gap p_{kk^{\prime}} := p_{k} - p_{k^{\prime}}.
\end{equation}

\begin{lem} \label{lem:ijswitch}  
  Assume the setting of Theorem \ref{thm:LdecBound} and that there are also no ties between the scores $\left\{\gamma_{ik}^{(n)}\right\}_{k=1}^E$ and biases $\left\{p_k^{(n)}\right\}_{k=1}^E$ themselves. Suppose two tokens $i$ and $j$ both remove the same expert $k^-$ and add the same expert $k^+$ between iterations $n$ and $n+1$, i.e., $k^-\in \rmvinds^{(n+1)}(i)\cap \rmvinds^{(n+1)}(j)$ and $k^+\in \addinds^{(n+1)}(i)\cap \addinds^{(n+1)}(j)$. Then, their score gaps relative to those experts satisfy
\[
\left|\gap\gamma_{k^+k^-}^{i}-\gap\gamma_{k^+k^-}^{j}\right|<2u.
\]
\end{lem}
\begin{proof}
Since $k^-\in \rmvinds^{(n+1)}(i)$ and $k^+\in \addinds^{(n+1)}(i)$, Lemma~\ref{lem:pairwise_bounds}a implies
\[
  -2u < \gamma_{ik^+}+p_{k^+}^{(n)}-\big(\gamma_{ik^-}+p_{k^-}^{(n)}\big) < 0.
\]
Writing $\gap\gamma_{k^+k^-}^{i} = \gamma_{ik^+}-\gamma_{ik^-}$ and $\gap p_{k^-k^+}^{(n)} = p_{k^-}^{(n)}-p_{k^+}^{(n)}$, this inequality is equivalent to
\[
  \gap p_{k^-k^+}^{(n)} - 2u < \gap\gamma_{k^+k^-}^{i} < \gap p_{k^-k^+}^{(n)}.
\]
The same interval constraint holds for token $j$. Since the upper and lower bounds are independent of $i,j$ and the interval has length at most $2u$, the result follows.
\end{proof}
Lemma \ref{lem:ijswitch} leads to an interesting implication: If we choose the step parameter $u$ to be smaller than half the minimum difference between any two \emph{distinct} score gaps, i.e.,
\[
  u < \bar{u}
  \quad\text{where}\quad
  \bar{u} := \frac{1}{2}\min_{k\ne k^{\prime}}\;\min_{i\ne j}\left|\gap\gamma_{kk^{\prime}}^{i}-\gap\gamma_{kk^{\prime}}^{j}\right|,
\]
and assume $\bar{u}>0$, then movements between entering-exiting pairs of experts become unique to a particular token in any two consecutive iterations.

\begin{prop}
\label{prop:ubar} \textbf{(Uniqueness of Token Movements)} Assume the setting of Theorem \ref{thm:LdecBound} with step-length $u<\bar{u}$ and where there are also no ties between the scores $\left\{\gamma_{ik}^{(n)}\right\}$ and biases $\left\{p_k^{(n)}\right\}$ themselves. Consider the update from iteration $n$ to $n+1$ and some fixed pair of experts $(k^-,k^+)$. Then, at most one token can simultaneously remove expert $k^-$ and add expert $k^+$. As a consequence, an expert's load cannot change by more than $(E-1)$ tokens between two consecutive iterations.
\end{prop}
\begin{proof}
Suppose, for contradiction, that between some iterations $n$ and $n+1$ two distinct tokens $i\neq j$ both remove the same expert $k^-$ and add the same expert $k^+$. Then Lemma~\ref{lem:ijswitch} implies
\[
  \big|\gap\gamma_{k^+k^-}^{i}-\gap\gamma_{k^+k^-}^{j}\big| < 2u.
\]
By the definition of $\bar{u}$, however,
\[
  \big|\gap\gamma_{k^+k^-}^{i}-\gap\gamma_{k^+k^-}^{j}\big| \ge 2\bar{u} > 2u,
\]
which is a contradiction. Hence, for each ordered pair $(k^-,k^+)$, at most one token can remove $k^-$ and add $k^+$ between iterations $n$ and $n+1$. There are $E-1$ possible origins (or destinations) for any fixed expert, so its load can increase or decrease by at most $(E-1)$ tokens in a single iteration.
\end{proof}

\subsubsection{Convergence to Approximate Balance Band}

Finally, we show that loads eventually converge to an {\it approximate balance band} where loads are within $(E-1)$ of the target load $L$. Specifically, the below Theorem \ref{thm:band_stability} shows that once an expert's load enters the band, it remains in that band for all subsequent iterations. Later, Theorem \ref{thm:band_entry} will show that all loads eventually enter the band in finite time.

\begin{thm}\label{thm:band_stability}
\textbf{(Approximate Balance Band Stability)} Assume the setting of Proposition~\ref{prop:ubar}. Then, once an expert's load enters the range $\left[L-\left(E-1\right),L+\left(E-1\right)\right]$, it remains in that range for all subsequent iterations.
\end{thm}
\begin{proof}
Fix an expert $k$ and suppose that for some iteration $n$,
\[
  A_k^{(n)} \in \big[L-(E-1),\,L+(E-1)\big].
\]
It suffices to show that $A_k^{(n+1)}$ also lies in this interval as all subsequent iterations then follow by induction.
Consider the following three cases:
\begin{itemize}
\item If $A_k^{(n)} = L$, then Proposition~\ref{prop:ubar} yields
\[
  \big|A_k^{(n+1)} - A_k^{(n)}\big| \le E-1,
\]
so $A_k^{(n+1)} \in [L-(E-1),\,L+(E-1)]$.
\item If $L+(E-1) \geq A_k^{(n)} > L$, then expert $k$ is overloaded at iteration $n$ and hence, by Theorem~\ref{thm:DSswitching}, it cannot be the entering expert in iteration $n+1$ in any change. Therefore,
\[
  A_k^{(n+1)} \le A_k^{(n)} \le L+(E-1).
\]
Combining this with the per-iteration bound from Proposition~\ref{prop:ubar} gives
\[
  A_k^{(n+1)} \ge A_k^{(n)} - (E-1) \ge L-(E-1),
\]
so $A_k^{(n+1)}\in [L-(E-1),\,L+(E-1)]$.
\item If $L > A_k^{(n)} > L-(E-1)$, then expert $k$ is underloaded at iteration $n$ and an analogous argument to the above case (but instead using the fact that $k$ is underloaded and so can not be the the exiting expert in any change) shows that $A_k^{(n+1)}\in [L-(E-1),\,L+(E-1)]$ as well.
\end{itemize}
This completes the proof.
\end{proof}

Now, we show the finite-time convergence of all expert loads to the approximate balance band.

\begin{lem}\label{lem:uniform_dominance}
\textbf{(Uniform Score Dominance Implies Load Dominance)}
Assume the setting of Theorem \ref{thm:LdecBound}. Fix an iteration $n$ and consider two fixed experts $k,k'$. If
\[
  \gamma_{ik}+p_k^{(n)} > \gamma_{ik'}+p_{k'}^{(n)}\qquad\forall i\in\{1,\dots,T\},
\]
then, $k'\in\alpha_n\left(\bi\right)$ implies $k\in\alpha_n\left(\bi\right)$ for any token $\bi$. Consequently, $A_k^{(n)}\ge A_{k'}^{(n)}$.
\end{lem}
\begin{proof}
Consider any fixed token $\bi$. If $k'\in\alpha_n\left(\bi\right)$, then $k'$ is among the $K$ largest values of $\left\{\gamma_{\bi\ell}+p_\ell^{(n)}\right\}_{\ell=1}^E$. Since $\gamma_{ik}+p_k^{(n)}$ is strictly larger than $\gamma_{ik'}+p_{k'}^{(n)}$  for all $i$ and there are no ties, expert $k$ must rank above $k'$ for token $\bi$. So, $k \in \topkind(\{\gamma_{\bi\ell}+p^{(n)}_\ell\}_\ell)$ and $k\in\alpha_n\left(\bi\right)$. Applying this argument across all tokens $i$ for which $k^\prime \in \alpha_n\left(i\right)$ yields $A_k^{(n)}\ge A_{k'}^{(n)}$.
\end{proof}

\begin{thm}\label{thm:band_entry}
\textbf{(Finite-Time Entry into the Approximate Balance Band)}
Assume the setting of Proposition~\ref{prop:ubar}. Then, for each expert $k$, there exists a finite iteration $N_k$ such that its load satisfies
\[
  A_k^{(N_k)} \in \big[L-(E-1),\,L+(E-1)\big].
\]
Moreover, letting $N:=\max_{k} N_k$, we have $A_k^{(m)} \in [L-(E-1),\,L+(E-1)]$ for every $m\ge N$.
\end{thm}
\begin{proof}
Fix an expert $k$. We prove that the load of expert $k$ must enter the band in finite time.

Suppose, for contradiction, that $A_k^{(n)}\notin [L-(E-1),\,L+(E-1)]$ for every $n\ge0$. Since $A_k^{(n)}$ is an integer, for each $n$ we must have bands in the ``over-band'' region $A_k^{(n)}\ge L+E$ or the ``under-band'' region $A_k^{(n)}\le L-E$. Moreover, by Proposition~\ref{prop:ubar}, $\big|A_k^{(n+1)}-A_k^{(n)}\big|\le E-1$, so expert $k$'s load cannot move between the over-band region to the under-band region without landing in the band. Hence, exactly one of the following cases holds for all $n\ge0$:
\[
  A_k^{(n)}\ge L+E \quad\text{for all }n
  \qquad\text{or}\qquad
  A_k^{(n)}\le L-E \quad\text{for all }n.
\]
We will show a contradiction in the the first, over-band case. The second case analogously follows a nearly identical argument.

Thus, we consider when $A_k^{(n)}\ge L+E$ for all $n$. So, $k$ is overloaded at every iteration, so $\sign(L-A_k^{(n)})=-1$ and the update \eqref{eq:alf_lb_update} gives $p_k^{(n+1)}=p_k^{(n)}-u$ for all $n$.

Next, every token is routed to exactly $K$ experts, so $\sum_{r=1}^E A_r^{(n)} = KT = EL$ for all $n$. Since $A_k^{(n)}>L$ for all $n$, at each iteration, there must exist at least one underloaded expert: otherwise, all $E$ loads would be at least $L$ with at least one strictly larger than $L$, forcing the sum to exceed $EL$. Because there are finitely many experts, there exists an expert $k^{\star}$ that is underloaded at infinitely many iterations, that is, $A_{k^{\star}}^{(n)}<L$ for infinitely many $n$.

For every $n$, the bias gap \eqref{eq:bias_gap} satisfies
\[
  \begin{aligned}
  \gap p_{k^{\star}k}^{(n+1)}-\gap p_{k^{\star}k}^{(n)}
  &= \left(p_{k^{\star}}^{(n+1)}-p_k^{(n+1)}\right)-\left(p_{k^{\star}}^{(n)}-p_k^{(n)}\right) \\
  &= \left(p_{k^{\star}}^{(n+1)}-p_{k^{\star}}^{(n)}\right)-\left(p_k^{(n+1)}-p_k^{(n)}\right) \\
  &= u\left(\sign\!\left(L-A_{k^{\star}}^{(n)}\right) - \sign\!\left(L-A_k^{(n)}\right)\right) \qquad \text{(by \eqref{eq:alf_lb_update})}\\
  &= u\left(\sign\!\left(L-A_{k^{\star}}^{(n)}\right)+1\right)\\
  &\ge \;0,
  \end{aligned}
\]
where, in the second-to-last line, we used $\sign\!\left(L-A_k^{(n)}\right)=-1$ since we assumed expert $k$ is overloaded at every iteration; and the last-line inequality holds because $\sign\!\left(\cdot\right)\in\{-1,0,1\}$.

Next, whenever $A_{k^{\star}}^{(n)}<L$, we have $\sign\!\left(L-A_{k^{\star}}^{(n)}\right)=1$ and thus $\gap p_{k^{\star}k}^{(n+1)}-\gap p_{k^{\star}k}^{(n)}=2u$. Since $k^{\star}$ is underloaded infinitely often, $\gap p_{k^{\star}k}^{(n)}\to+\infty$ and, because it is nondecreasing, there exists $n_1$ such that
\[
  \gap p_{k^{\star}k}^{(n)} > 1 \qquad\forall n\ge n_1.
\]
For any token $i$ and any $n\ge n_1$,
\[
  (\gamma_{ik^{\star}}+p_{k^{\star}}^{(n)}) - (\gamma_{ik}+p_k^{(n)})
  = (\gamma_{ik^{\star}}-\gamma_{ik}) + (p_{k^{\star}}^{(n)}-p_k^{(n)})
  > -1 + 1 = 0,
\]
where we used $\gamma_{ik^{\star}},\gamma_{ik}\in(0,1)$ (Section \ref{sec:deterministic_assumptions}). Hence, for all tokens $i$ and all $n\ge n_1$,
\[
  \gamma_{ik^{\star}}+p_{k^{\star}}^{(n)} > \gamma_{ik}+p_k^{(n)}.
\]
By Lemma~\ref{lem:uniform_dominance}, this implies $A_{k^{\star}}^{(n)}\ge A_k^{(n)}\ge L+E$ for all $n\ge n_1$, so $k^{\star}$ is overloaded for all $n\ge n_1$. This contradicts that $k^{\star}$ is underloaded infinitely often.

Therefore, the over-band case is impossible. The under-band case follows analogously.

We conclude that expert $k$ must enter the band $[L-(E-1),\,L+(E-1)]$ at some finite time $N_k$. Since $k$ was arbitrary, this holds for all experts. Letting $N:=\max_k N_k$ yields simultaneous band entry, and Theorem~\ref{thm:band_stability} gives permanence thereafter.
\end{proof}

When the number of tokens is much larger than the number of experts ($T\gg E$), the deviation of $(E-1)$ from the perfectly balanced load $L=KT/E$ is negligible. This demonstrates DeepSeek's ALF-LB desirability under this section's stylized, deterministic setting. The next section will consider a more realistic stochastic setting.

\section{Stochastic Analysis via Online Optimization}\label{sec:oco_analysis}

In practice, the affinity scores $\gamma_{ik}^{(n)}$ evolve during training. Thus, we generalize the previously considered Lagrangian \eqref{eq:L} by now {\bf assuming here in Section \ref{sec:oco_analysis}} that the scores $\gamma_{ik}^{(n)}$ are {\it stochastic} and drawn from expert-dependent distributions. In particular, at iteration $n$, we assume a random affinity score $\gamma_{ik}^{(n)} \in (0,1)$ is observed for each token $i\in\{1,\dots,T\}$ and expert $k\in\{1,\dots,E\}$.  The algorithm updates a shift vector $p^{(n)}\in\R^E$ and, for each token $i$, selects the $K$ experts with the largest values of $\gamma_{ik}^{(n)}+p^{(n)}_k$, where $p^{(n)}_k$ is the $k$-th coordinate of $p^{(n)}$. 

Using the notation of Section \ref{sec:primal_dual_alflb}, we note that this section will consider a step-size choice of $\epsilon_k^{(n)} = u/n$ instead of the $u/|L - A_k^{(n)}|$ step-size chosen by \cite{wang2024auxiliary}. This is because, when affinity scores become stochastic and time-varying, analyzing the coordinate-dependent $u/|L - A_k^{(n)}|$ step-size sequence becomes technically intricate. In contrast, the diminishing and coordinate-independent $u/n$ step-size connects directly with ideas from online convex analysis \citep[Section 3.3.1]{hazan2016introduction} leading to cleaner theoretical insights. This adjustment maintains experimental relevance and practicality: Figure \ref{fig:convergence_behavior} and Table \ref{tab:final_compare} demonstrate that using the $u/n$ step-size is comparable in effectiveness as using the original $u/|L - A_k^{(n)}|$ step-size. In fact, Table \ref{tab:final_compare} shows that the $u/n$ step-size leads to a slightly {\it better} load balancing performance than the $u/|L - A_k^{(n)}|$ step-size at the cost of a slightly higher validation loss.

\subsection{Notation}\label{sec:oco_notation}

For any vector $z\in\R^E$, define $\topkind(z)\subseteq\{1,\dots,E\}$ to be the set of indices of the $K$ largest coordinates of $z$ with ties broken arbitrarily. For round $n$ and token $i$, denote $\Gamma^{(n),i} := \left(\gamma_{i1}^{(n)},\dots,\gamma_{iE}^{(n)}\right)\in\R^E$. Routing at round $n$ sends token $i$ to the experts in $\topkind\!\left(\Gamma^{(n),i}+p^{(n)}\right)$.

Fix an integer $K\in\{1,\dots,E\}$. We frame this as minimizing the per-round online loss, which corresponds to the dual online objective:
\begin{equation}
  f^{(n)}(p) 
  = \sum_{i=1}^T \; \sum_{k\in \topkind\!\left(\Gamma^{(n),i}+p\right)} \left(\gamma_{ik}^{(n)} + p_k\right) 
  - L \sum_{k=1}^E p_k,
  \label{eq:OCOobjective}
\end{equation}
where $L:=KT/E$ is the desired per-expert load for Top-K routing. The $k$-th component of the loss gradient $g^{(n)}(p) \in \R^E$ is given by
\begin{equation}
  g^{(n)}_k(p):=\nabla_k f^{(n)}(p) = A_k^{(n)}(p) - L,
  \label{eq:grad_f}
\end{equation}
where $A_k^{(n)}(p):= \left|\left\{i: k \in \topkind\!\left(\Gamma^{(n),i}+p\right)\right\}\right|$ counts the number of tokens for which expert $k$ lies in the per-token Top-K set under shifts $p$.
The online dual update corresponding to \eqref{eq:dualupdate} can then be rewritten as $p_{k}^{\left(n+1\right)} \gets p_{k}^{\left(n\right)} - \epsilon_{k}^{\left(n\right)}g_k^{(n)}(p^{(n)})$.

\subsection{Distributional Assumptions}\label{sec:distassume}
{\bf Assume for the remainder of this section that} each affinity $\gamma_{ik}^{(n)}$, for a fixed expert $k$, is drawn independently\footnote{This independence assumption is a stylization: various mechanisms (attention, layer norm, etc.) in earlier layers could create dependencies between token embeddings. However, it gives us a starting point for building a tractable, baseline theory. Furthermore, Figure \ref{fig:scores_histograms} demonstrates that the distributions of $\gamma_{ik}^{(n)}$ remain stable and well-behaved throughout training on DeepSeekMoE-1B models, which indicates that independence is empirically well-founded as an approximating simplification at least at the marginal distribution level.} from a distribution that
\begin{itemize}
\item depends only on the expert $k$,
\item has bounded support on $(0,1)$, and
\item has a probability density function (pdf) $\varphi_k$ that is upper bounded by some expert-independent constant.
\end{itemize}
Thus, for fixed $k$, the vectors of random affinity scores $\Gamma^{(n),i} = \left(\gamma_{ik}^{(n)}\right)_{k=1}^E$ for the $i$-th token in the $n$-th iteration are i.i.d across $i$ and $n$. While this assumption may seem strong, our experiments in Figure \ref{fig:scores_histograms} from training 1B-parameter DeepSeekMoE models suggest that it is close to reality.

\begin{figure}[!ht]
  \centering
  \begin{subfigure}[b]{0.45\textwidth}
      \centering
      \includegraphics[width=\textwidth]{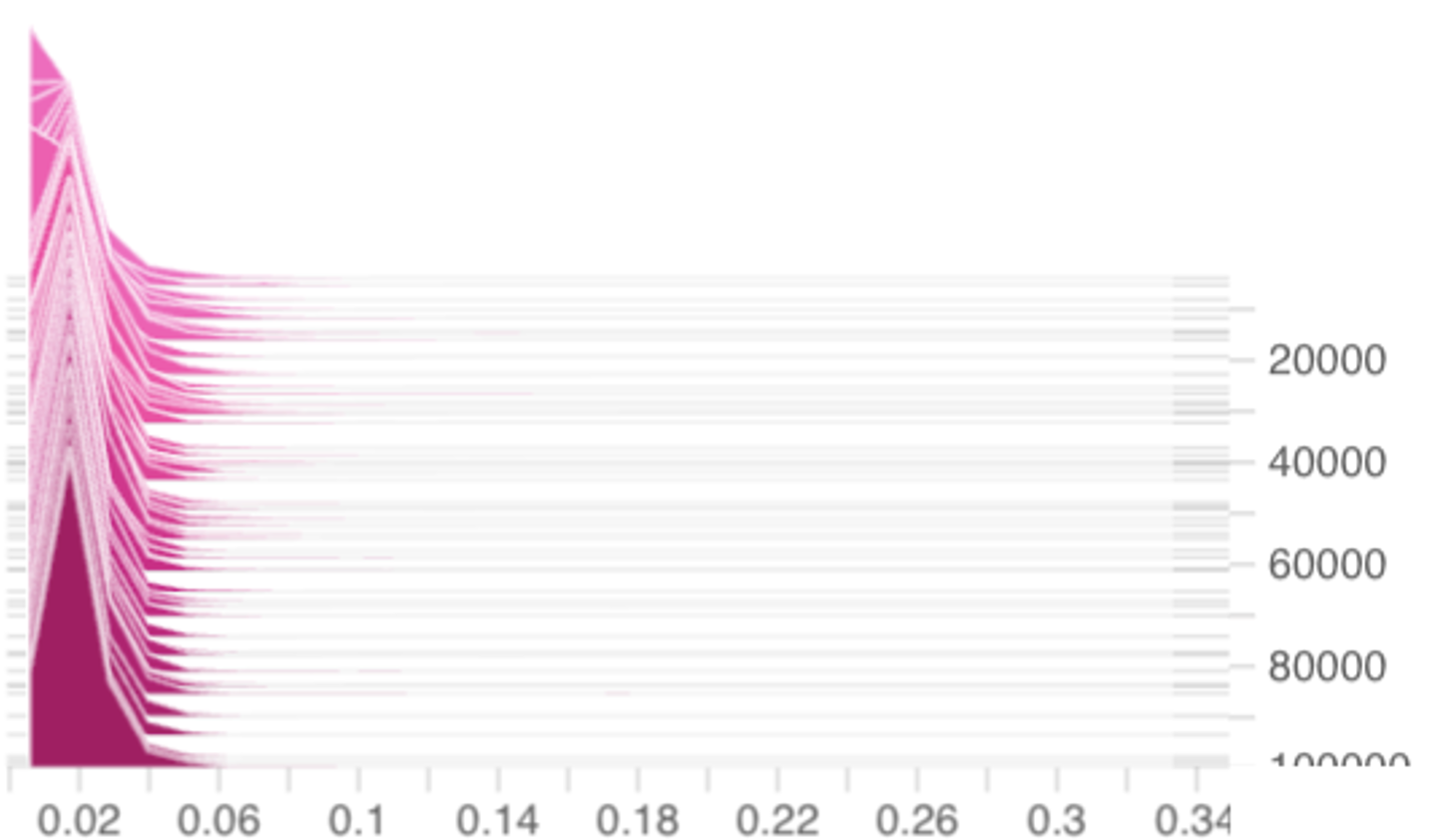}
      \caption{$u/|L-A_k^{(n)}|$ Step-Size}
  \end{subfigure}
  \hfill
  \begin{subfigure}[b]{0.45\textwidth}
      \centering
      \includegraphics[width=\textwidth]{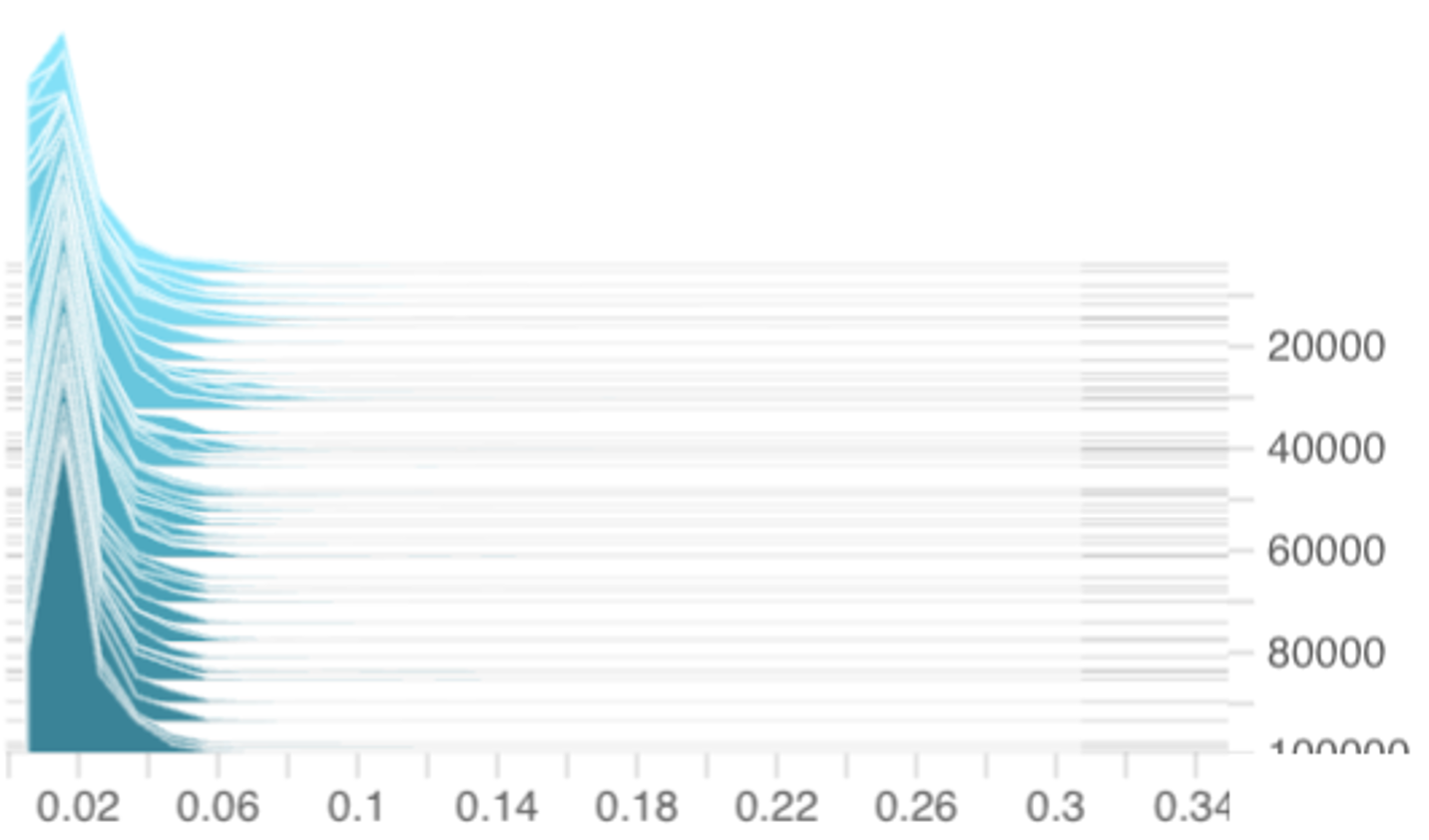}
      \caption{$u/n$ Step-Size}
  \end{subfigure}
  \hfill
  \begin{subfigure}[b]{0.45\textwidth}
    \centering
    \includegraphics[width=\textwidth]{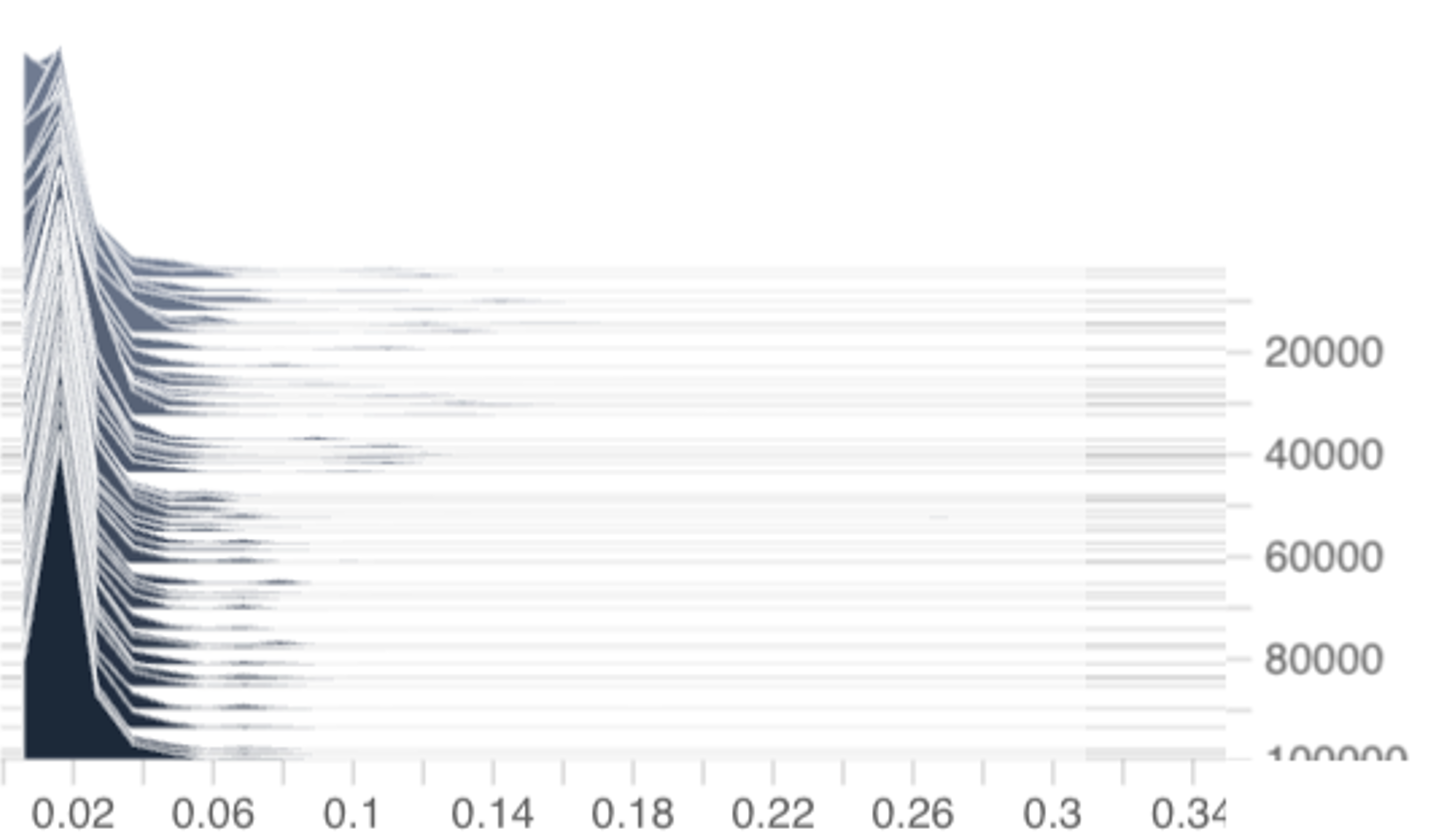}
    \caption{$u/\sqrt{n}$ Step-Size}
  \end{subfigure}
  \hfill
  \begin{subfigure}[b]{0.45\textwidth}
    \centering
    \includegraphics[width=\textwidth]{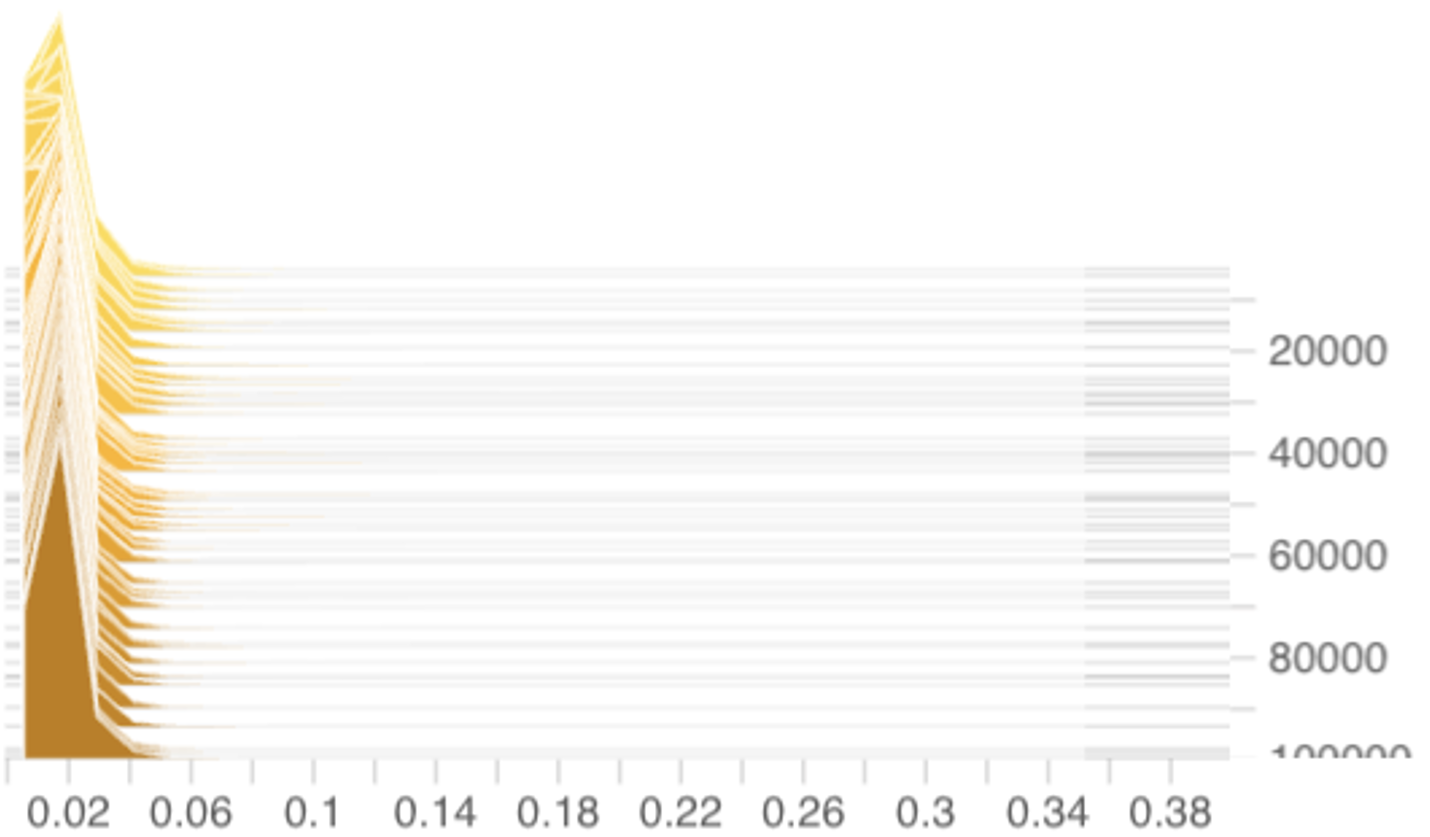}
    \caption{Auxiliary Loss}
  \end{subfigure}
  \caption{Time-lapse histograms of the marginal distributions of $\gamma^{(n)}_{ik}$ during the training of \mbox{1B-parameter} DeepSeekMoE models using different choices of step-size (Section \ref{sec:primal_dual_alflb}). Experimental details in Section \ref{sec:experiments}. \label{fig:scores_histograms}}
\end{figure}

\subsection{Online Loss Gradient Analysis}

\subsubsection{Unbiasedness}

It is easy to check that the loss $f^{(n)}$ is convex. Thus, the expected loss $\mathbf{f}(p)=\E\left[\left.f^{(n)}(p)\right| p\right]$ is also convex.  We first show that the loss gradient $g^{(n)}(p)$ is an unbiased estimator of $\nabla \mathbf{f}(p)$ with an explicit form expression.

\begin{prop}[Unbiased Stochastic Gradient]\label{prop:unbiased}
For any fixed $p\in\R^E$,
\[
  \E\left[\left. g^{(n)}(p)\right| p\right] \,=\, \nabla \mathbf{f}(p) \,=\, T\,\pi(p) - L\,\ones,
\]
where $\ones$ is the ones-vector, and $\pi(p) \in \R^E$ is the selection probabilities vector with $k$-th coordinate
\[
  \pi_k(p) \,:=\, \Pr\!\left(k\in \topkind\!\left(\Gamma+p\right)\right),
\]
for some generic affinities $\Gamma \stackrel{d}{=} \Gamma^{(n),i}$. Necessarily, $\sum_k \pi_k(p)=K$ almost surely.
\end{prop}
\begin{proof}
For a given token $i$, let 
\begin{equation}
X_{ik}(p):=\ind\left\{k\in\topkind\left(\Gamma^{(n),i} + p\right)\right\}.\label{eq:Xik}
\end{equation}
Then, $\E\left[X_{ik}(p)\mid p\right]=\pi_k(p)$ by definition and the $k$-th expert loads are $A_k^{(n)}(p)=\sum_{i=1}^T X_{ik}(p)$. Hence, $\E\left[A_k^{(n)}(p)\mid p\right]=T\,\pi_k(p)$. Therefore,
\begin{align*}
\E\left[g^{(n)}(p)\mid p\right]
  &= \E\left[\nabla f^{(n)}(p)\mid p\right] \\
  &= \left(\E\left[A_1^{(n)}(p)\mid p\right]-L,\dots,\E\left[A_E^{(n)}(p)\mid p\right]-L\right) \quad \text{{by Eq. \ref{eq:grad_f}}}\\
  &= T\,\pi(p)-L\,\ones.
\end{align*}
Since the distribution of $\Gamma^{(n),i}$ is independent of $i$, observe that 
\[\mathbf{f}(p)=T\,\E\left[\sum_{k\in \topkind(\Gamma+p)}(\Gamma_k+p_k)\right]-L\sum_k p_k.\]
To calculate $\nabla \mathbf{f}$, we can move the differentiation into the expectation via the dominated convergence theorem: the $\Gamma_{k}$ have continuous densities so ties occur with probability zero; thus, for almost every realization, the partial derivative $\partial_{p_k}\sum_{m\in \topkind(\Gamma+p)}(\Gamma_m+p_m)$ exists and equals $\ind\{k\in \topkind(\Gamma+p)\}\le 1$, yielding $\partial_{p_k}\,\E\left[\sum_{m\in \topkind(\Gamma+p)}(\Gamma_m+p_m)\right]=\Pr\{k\in \topkind(\Gamma+p)\}=\pi_k(p)$. The desired result follows.
\end{proof}
Using Proposition \ref{prop:unbiased}, we can compute the variance and second moment of $g^{(n)}(p)$.
\begin{prop}(Variance and 2nd Moment) \label{prop:var2ndmom} The variance and second moment of $g^{(n)}(p)$ are given by
\[
  \Var\left[\left.g^{(n)}(p) \right| p\right] = T\left(K-\sum_{k=1}^E \pi_k(p)^2\right),
\]
\[
  \E\!\left[\left.\|g^{(n)}(p)\|^2 \right| p\right] = T^2\left(\sum_{k=1}^E \pi_k(p)^2 - \tfrac{K^2}{E}\right) + T\left(K-\sum_{k=1}^E \pi_k(p)^2\right).
\]
\end{prop}
\begin{proof}
Using Equation \eqref{eq:grad_f} and Proposition \ref{prop:unbiased},
\begin{align*}
  \Var\left(\left. g^{(n)}(p)\right| p \right) &= \E\!\left[\|g^{(n)}(p)-\nabla \mathbf{f}(p)\|^2\mid p\right]\\
  &= \E\!\left[\|\pi(p)-A^{(n)}(p)\|^2\mid p\right]
\end{align*}
where $A^{(n)}(p) = \left(A_1^{(n)}(p), \dots, A_E^{(n)}(p)\right)$ is the vector of expert loads.
Let $X_i(p) \in \{0,1\}^E$ denote the assignment vector for token $i$ with components as in \eqref{eq:Xik}. Then, $A^{(n)}(p) = \sum_{i=1}^T X_{i}(p)$ and $\E[X_i(p)] = \pi(p)$. So,
\begin{align*}
  \Var\left(\left. g^{(n)}(p)\right| p \right) &=  \E\!\left[\left.  \left\|\sum_{i=1}^T \pi(p)- X_{i}(p)\right\|^2 \right| p \right]\\
  &= \sum_{i=1}^T \E\!\left[\|\pi(p)-X_i(p)\|^2\mid p\right] \quad \text{by independence across $i$}\\
  &= T\sum_{k=1}^E \Var(X_{1k}\mid p) \quad \text{by identical distribution across $i$}\\
  &= T\left(K-\sum_{k=1}^E \pi_k(p)^2\right),
\end{align*}
since $\sum_k X_{ik}=K$ a.s. and $\Var(X_{1k}\mid p)=\pi_k(p)\left(1-\pi_k(p)\right)$. 
Finally, decomposing the second moment and applying Proposition \ref{prop:unbiased} gives
\begin{align*}
  \E\!\left[\|g^{(n)}(p)\|^2\mid p\right]
  &= \|\nabla\mathbf{f}(p)\|^2 + \E\!\left[\|g^{(n)}(p)-\nabla\mathbf{f}(p)\|^2\mid p\right] \\
  &= T^2\left(\sum_{k=1}^E \pi_k(p)^2 - \tfrac{K^2}{E}\right) + T\left(K-\sum_{k=1}^E \pi_k(p)^2\right).
\end{align*}
This completes the proof.
\end{proof}
Proposition \ref{prop:var2ndmom} will be useful later to prove Theorem \ref{thm:logarithmic_regret}.

\subsection{Second-Order Analysis of Expected Loss}\label{sec:2ndAnalysis}

In the following Sections \ref{sec:2ndAnalysis}-\ref{sec:strong_convexity}, we show that the expectation of the Top-K objective is {\it strongly convex} with respect to $p$ updates under certain (realistic) assumptions. The strong convexity then allows us to show a logarithmic regret bound in Theorem \ref{thm:logarithmic_regret}. Without strong convexity, it is routine to verify that the regret bound is at best $\mathrm{O}(\sqrt{N})$ without additional assumptions. The next lemma characterizes the second directional derivative of the expected objective.

\begin{prop}[Second Directional Derivative]\label{prop:second_directional}
Let $\Gamma=(\Gamma_1,\dots,\Gamma_E)$ be a random affinity vector in $\R^E$ with the properties in Section \ref{sec:distassume}. For biases $p\in\R^E$ define
\[
  \mathbf{F}_K(p) \;=\; \mathbb{E}\left[\,\sum_{k\in \topkind\!\left(\Gamma+p\right)}(\Gamma_k + p_k)\,\right],
\]
and let $\varphi_k$ and $\Phi_k$ denote the density and CDF of $\Gamma_k$, respectively. Then, for any direction $\delta\in\R^E$, its second directional derivative at $p$ is given by the formula
\begin{equation}
  D^2\mathbf{F}_K(p)[\delta,\delta] \;=\; \sum_{k<\ell} w^{(K)}_{k\ell}(p)\,(\delta_k-\delta_\ell)^2,
  \label{eq:D2weighted}
\end{equation}
where the symmetric edge weights are
\[
  w^{(K)}_{k\ell}(p) \;=\; \int_{-\infty}^{\infty} \! \varphi_k(v-p_k)\,\varphi_\ell(v-p_\ell)\,B^{(K-1)}_{k,\ell}(v;p)\,dv,\qquad w^{(K)}_{k\ell}(p)\ge 0,
\]
with
\[
  B^{(K-1)}_{k,\ell}(v;p) \,=\, \sum_{\substack{S\subseteq [E]\setminus\{k,\ell\}\\ |S|=K-1}} \; \prod_{j\in S}\Phi^c_j(v-p_j)\, \prod_{m\in [E]\setminus(\{k,\ell\}\cup S)}\!\Phi_m(v-p_m).
\]
\end{prop}
\begin{proof}
For some fixed argument $\gamma\in[0,1]^E$, define the function
\[\bar f_{p,K}(\gamma)=\sum_{k\in \topkind\!\left(\gamma+p\right)}(\gamma_k+p_k),\]
so that $\mathbf{F}_K(p)=\int \bar f_{p,K}(\gamma)\,\varphi(\gamma)\,d\gamma$, where $\varphi(\gamma)=\prod_{k=1}^E \varphi_k(\gamma_k)$ is the joint density of $\Gamma$.  For $t \in \R$, define $p(t):=p+t\,\delta$ and $\widetilde{\mathbf{F}}_K(t):=\mathbf{F}_K\left(p(t)\right)$.  Since ties occur with probability zero, $\mathbf{F}_K$ is a.s. differentiable with gradient $\nabla \mathbf{F}_K(p)=\pi(p)$. Hence, the chain rule gives
\[
  \widetilde{\mathbf{F}}_K'(0) \,=\, \sum_{k=1}^E \delta_k\,\pi_k(p).
\]
We will next compute $\widetilde{\mathbf{F}}_K''(0)$. For each $k$, using independence and conditioning on $\Gamma_k=v$,
\begin{align*}
  \pi_k(p)
  &= \Pr\!\left(k\in \topkind\!\left(\Gamma+p\right)\right) \\
  &= \int_0^1 \varphi_k(v)\,\Pr\!\left( k\in \topkind\!\left(\Gamma+p\right)\,\middle|\,\Gamma_k=v \right) dv \\
  &= \int_0^1 \varphi_k(v)\,\Pr\!\left( \left|\{j\ne k: \Gamma_j+p_j>v+p_k\}\right| \le K{-}1 \right) dv \\
  &= \int_0^1 \varphi_k(v)\,\sum_{r=0}^{K-1} \Pr\!\left( \left|\{j\ne k: \Gamma_j+p_j>v+p_k\}\right| = r \right) dv \\
  &= \int_0^1 \varphi_k(v)\,\sum_{r=0}^{K-1}\; \sum_{\substack{S\subseteq [E]\setminus\{k\}\\ |S|=r}}\Pr\!\left( \forall j{\in}S: \Gamma_j{+}p_j>v{+}p_k \ \mathrm{and} \ \forall m \notin S{\cup}\{k\}: \Gamma_m{+}p_m\le v{+}p_k \right) dv \\
  &= \int_0^1 \varphi_k(v)\,\sum_{r=0}^{K-1}\; \sum_{\substack{S\subseteq [E]\setminus\{k\}\\ |S|=r}} \prod_{j\in S}\Phi^c_j(v-p_j+p_k)\,\prod_{m\notin S\cup\{k\}} \Phi_m(v-p_m+p_k)\,dv.
\end{align*}
where $\Phi^c_j(\cdot) := 1-\Phi_j(\cdot)$ and $S$ represents possible index sets within the top-K components of $\Gamma + p$ that are also larger than $v+p_k$. For notational convenience, set
\[\theta_{jk}(v,t) := v - (p_j - p_k) - t\,(\delta_j - \delta_k).\]
Then,
\begin{align*}
  \pi_k\left(p(t)\right)
  =\; &\int_0^1 \varphi_k(v)
      \sum_{r=0}^{K-1}
      \sum_{\substack{S\subseteq [E]\setminus\{k\}\\ |S|=r}}
      \left[
        \prod_{j\in S}\Phi^c_j\left(\theta_{jk}(v,t)\right)
        \prod_{m\notin S\cup\{k\}} \Phi_m\left(\theta_{mk}(v,t)\right)
      \right] dv.
\end{align*}
Consider the integrand 
\begin{equation}
\varphi_k(v)\;\sum_{r=0}^{K-1}\; \sum_{\substack{S\subseteq [E]\setminus\{k\}\\ |S|=r}} \prod_{j\in S} \Phi^c_j\left(\theta_{jk}(v,t)\right)\;\prod_{m\notin S\cup\{k\}} \Phi_m\left(\theta_{mk}(v,t)\right).
\label{eq:integrand_lem}
\end{equation}
For each $j$, because $\Phi_j$ is differentiable everywhere on $\R$ except (possibly) at $0$ or $1$, the derivative of the integrand \eqref{eq:integrand_lem} with respect to $t$ at $t=0$ exists for all but finitely many $v\in (0,1)$. Thus, the integrand \eqref{eq:integrand_lem} is differentiable at $t=0$ for almost all $v\in (0,1)$.

Next, observe that for each fixed $S\subseteq [E]\setminus\{k\}$, the product in \eqref{eq:integrand_lem} has the a.e. derivative
\begin{multline*}
  \frac{d}{dt}\,\left[\prod_{j\in S} \Phi^c_j\left(\theta_{jk}(v,t)\right)\;\prod_{m\notin S\cup\{k\}} \Phi_m\left(\theta_{mk}(v,t)\right)\right] \\
  = \sum_{\ell\notin S\cup\{k\}} \underbrace{(\delta_k-\delta_\ell)\,
    \varphi_\ell\left(\theta_{\ell k}(v,t)\right)
    \prod_{j\in S} \Phi^c_j\left(\theta_{jk}(v,t)\right)
    \prod_{m\notin S\cup\{k,\ell\}} \Phi_m\left(\theta_{mk}(v,t)\right)}_{:=\ \Xi^{+}_{S,\ell}} \\
    - \sum_{\ell\in S} \underbrace{(\delta_k-\delta_\ell)\,\varphi_\ell\left(\theta_{\ell k}(v,t)\right)
    \prod_{j\in S\setminus\{\ell\}} \Phi^c_j\left(\theta_{jk}(v,t)\right)
    \prod_{m\notin S\cup\{k\}} \Phi_m\left(\theta_{mk}(v,t)\right)}_{:= \ \Xi^{-}_{S,\ell}}.
\end{multline*}
This leads to a telescoping cancellation across $r$ in the integrand \eqref{eq:integrand_lem}. Specifically, for each fixed $r < K{-}1$ and $\ell$, every index set $S_r$ such that $|S_r|=r$ corresponds to another index set $S_{r+1}^\ell$ such that $|S_{r+1}^\ell|=r+1$ and $S_{r+1}^\ell=S_r\cup\{\ell\}$. It is easy to check that $\Xi^{+}_{S_{r}, \ell}=\Xi^{-}_{S_{r+1}^\ell, \ell}$.

So, the sum in \eqref{eq:integrand_lem} telescopes over $r$ except at the $r=K{-}1$ boundary where $\Xi^{+}_{S_{K-1}, \ell}$ has no corresponding ``$\Xi^{-}_{S_{K}^\ell, \ell}$'' term to cancel with. Hence, for almost all fixed $v \in (0,1)$, 
\begin{align*}
&~\frac{d}{dt} \left.\left[\varphi_k(v)\;\sum_{r=0}^{K-1}\; \sum_{\substack{S\subseteq [E]\setminus\{k\}\\ |S|=r}} \prod_{j\in S} \Phi^c_j\left(\theta_{jk}(v,t)\right)\;\prod_{m\notin S\cup\{k\}} \Phi_m\left(\theta_{mk}(v,t)\right)\right]\right|_{t=0} \\
&= \varphi_k(v) \sum_{\ell\ne k} (\delta_k-\delta_\ell)\,
    \varphi_\ell\left(\theta_{\ell k}(v,0)\right)\,
    \sum_{\substack{S\subseteq [E]\setminus\{k,\ell\}\\ |S|=K-1}} \prod_{j\in S} \Phi^c_j\left(\theta_{jk}(v,0)\right)\;\prod_{m\notin S\cup\{k,\ell\}} \Phi_m\left(\theta_{mk}(v,0)\right).
\end{align*}
Next, for any non-zero $h \in \R$, it is routine to check that the assumptions in Section \ref{sec:distassume} imply that the integrand of the following is uniformly bounded:
\begin{align*}
&~\frac{1}{h} \left[\pi_k\left(p(h)\right) - \pi_k\left(p(0)\right)\right] \\
=&~ \int_0^1 \frac{1}{h}\left[\varphi_k(v)\;\sum_{r=0}^{K-1}\; \sum_{\substack{S\subseteq [E]\setminus\{k\}\\ |S|=r}} \prod_{j\in S} \Phi^c_j\left(\theta_{jk}(v,h)\right) \prod_{m\notin S\cup\{k\}} \Phi_m\left(\theta_{mk}(v,h)\right)\right.\\
& \left.\hspace{1.2cm} - \varphi_k(v)\;\sum_{r=0}^{K-1}\; \sum_{\substack{S\subseteq [E]\setminus\{k\}\\ |S|=r}} \prod_{j\in S} \Phi^c_j\left(\theta_{jk}(v,0)\right)\;\prod_{m\notin S\cup\{k\}} \Phi_m\left(\theta_{mk}(v,0)\right)\right] dv
\end{align*}
Therefore, by the dominated convergence theorem, we have that 
 $\pi_k\left(p(t)\right)$ is differentiable at $t=0$, with the derivative being given (after a change of variables) by
\begin{multline*}
  D\pi_k(p)[\delta] \\
  = \sum_{\ell\ne k} (\delta_k-\delta_\ell)
    \int_{-\infty}^{\infty} \varphi_k(v-p_k)\,\varphi_\ell\left(v-p_\ell\right)
    \sum_{\substack{S\subseteq [E]\setminus\{k,\ell\}\\ |S|=K-1}}
     \prod_{j\in S} \Phi^c_j(v-p_j)
    \prod_{m\notin S\cup\{k,\ell\}} \Phi_m(v-p_m)
    \,dv
\end{multline*}
Finally, by symmetry,
\begin{align*}
  D^2\mathbf{F}_K(p)[\delta,\delta]
  &= \frac{d}{dt}\left.\sum_{k=1}^E \delta_k\,\pi_k(p+t\delta)\right|_{t=0}
   \,=\, \sum_{k=1}^E \delta_k\,D\pi_k(p)[\delta] \\
  &= \sum_{k=1}^E \sum_{\ell\ne k} \delta_k(\delta_k-\delta_\ell)\,w^{(K)}_{k\ell}(p)
   \,=\, \sum_{1\le k<\ell\le E} w^{(K)}_{k\ell}(p)\,(\delta_k-\delta_\ell)^2,
\end{align*}
which is exactly \eqref{eq:D2weighted}.
\end{proof}

\subsection{Experimentally-Realistic Assumptions on $p$} \label{sec:prac_calK}
Observe that the TopKInd decision of the MoE router is invariant to adding the same constant to all coordinates of $p$. Motivated by this, we define the zero-sum subspace
\[\calZ := \left\{ z \in \R^E : \sum_{k=1}^E z_k = 0 \right\},\]
where $\calZ$ is the linear subspace orthogonal to the all-ones vector. Thus, we assume the following about ALF-LB for some update direction $\delta^\prime$:
\begin{equation}
 p^{(n+1)} \gets \proj_{\calZ}\left(p^{(n)} - \delta^\prime\right).
 \label{eq:projKassumption}
\end{equation}

\begin{rem}[{\bf Practicality of $\calZ$ Assumptions}]\label{rem:prac_calK}
The assumption \eqref{eq:projKassumption} is not artificial; it arises naturally from the problem definition:
\begin{itemize}
    \item \textbf{Zero-sum gradients.} Since $\sum_k A_k^{(n)}(p){=}TK$, the components of the gradient \eqref{eq:grad_f} sum to zero:
    \[
    \sum_k \nabla_k f^{(n)}(p) = \sum_k (A_k^{(n)}(p) - L) = TK - EL = 0.
    \]
    Thus, any update of the form
    \[
    p^{(n+1)} \gets p^{(n)} - \epsilon^{(n)} \nabla f^{(n)}(p^{(n)})
    \]
    automatically preserves $p^{(n+1)} \in \calZ$ as long as $p^{(n)} \in \calZ$. In practice, we initialize with $p^{(0)} = 0$, so the projection in \eqref{eq:projKassumption} is just the identity mapping.
    \item \textbf{Explicit $\calZ$-projection with per-coordinate step-sizes.} In the more general case where heterogeneous step-sizes $\epsilon_k^{(n)}$ are used across coordinates,
    \[
      p^{(n+1)} \gets p^{(n)} - \left(\epsilon_1^{(n)} g^{(n)}_{1}, \dots, \epsilon_E^{(n)} g^{(n)}_{E}\right),
    \]
    the updated $p^{(n+1)}$ may not reside in $\calZ$. (In fact, the difference between per-coordinate step-sizes and homogeneous step-sizes can be seen in Figure \ref{fig:bias_histograms} where the per-coordinate $\epsilon_k^{(n)} = u/|L-A_k^{(n)}|$ step-size results in a bias distribution that shifts rightward over time while the homogeneous $\epsilon^{(n)} = u/n$ and $\epsilon^{(n)} = u/\sqrt{n}$ step-sizes result in bias distributions that stay centered around zero.) However, in this per-coordinate step-size case, it is well-known that the projection onto $\calZ$ is equivalent to subtracting the componentwise mean:
    \[
      \proj_{\calZ}(p) = p - \left(\tfrac{1}{E}\sum_{k=1}^E p_k\right)\ones,
    \]
    which is computationally negligible.
\end{itemize}
\end{rem}

\begin{figure}[!ht]
  \centering
  \begin{subfigure}[b]{0.45\textwidth}
      \centering
      \includegraphics[width=\textwidth]{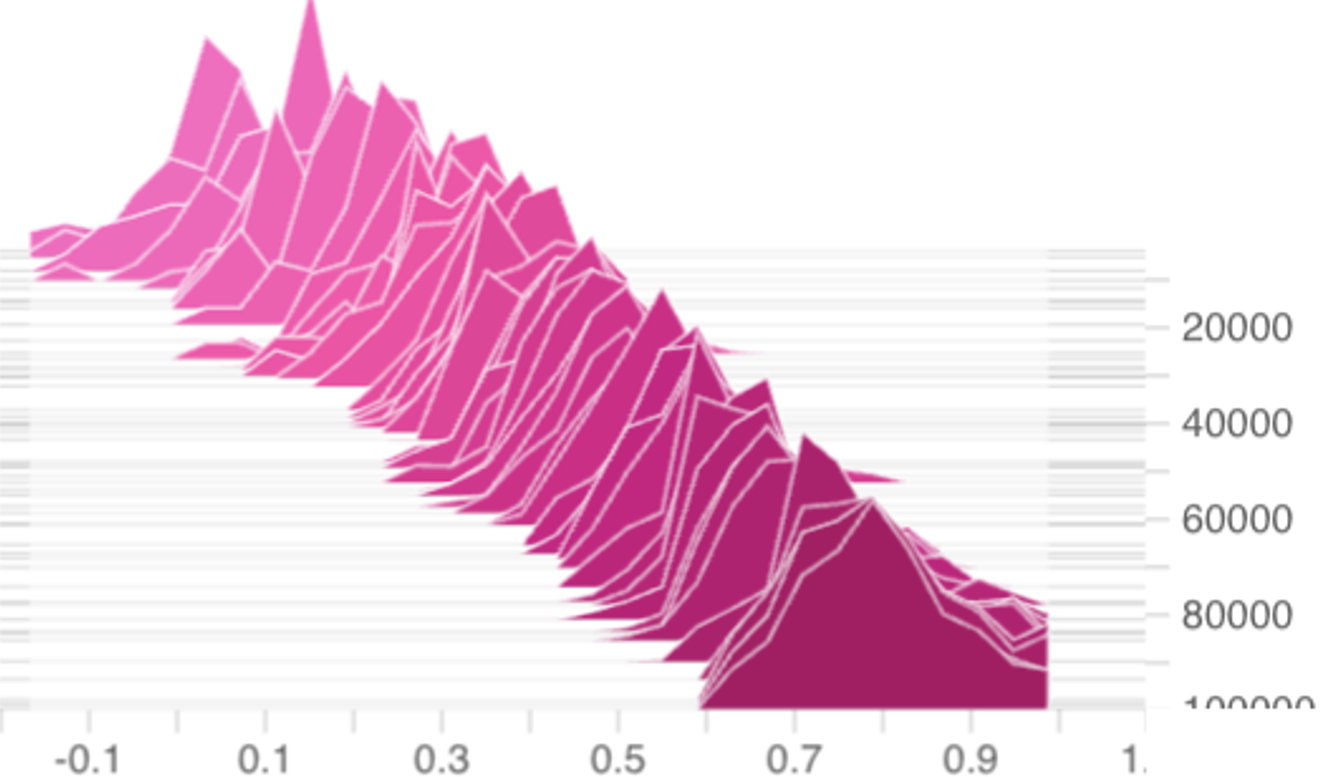}
      \caption{$u/|L-A_k^{(n)}|$ Step-Size}
  \end{subfigure}
  \hfill
  \begin{subfigure}[b]{0.45\textwidth}
      \centering
      \includegraphics[width=\textwidth]{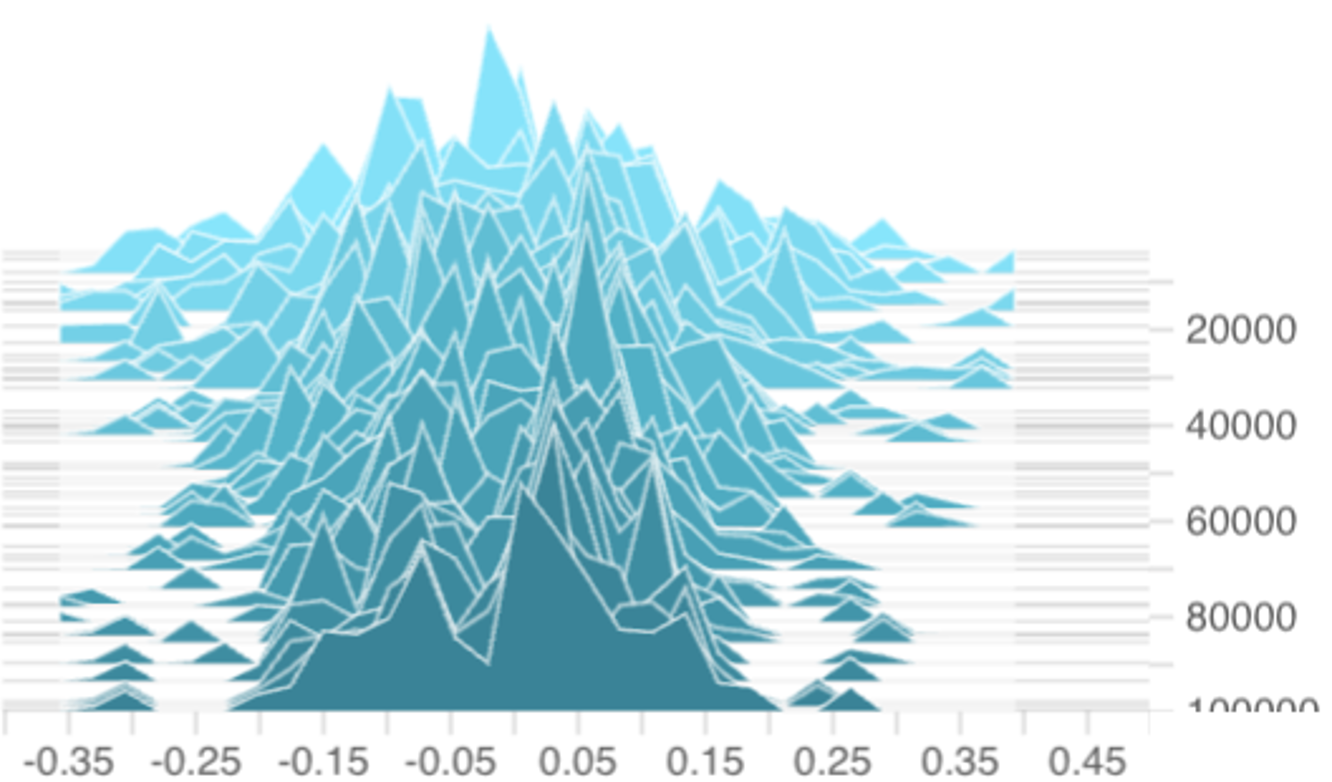}
      \caption{$u/n$ Step-Size}
  \end{subfigure}
  \hfill
  \begin{subfigure}[b]{0.45\textwidth}
    \centering
    \includegraphics[width=\textwidth]{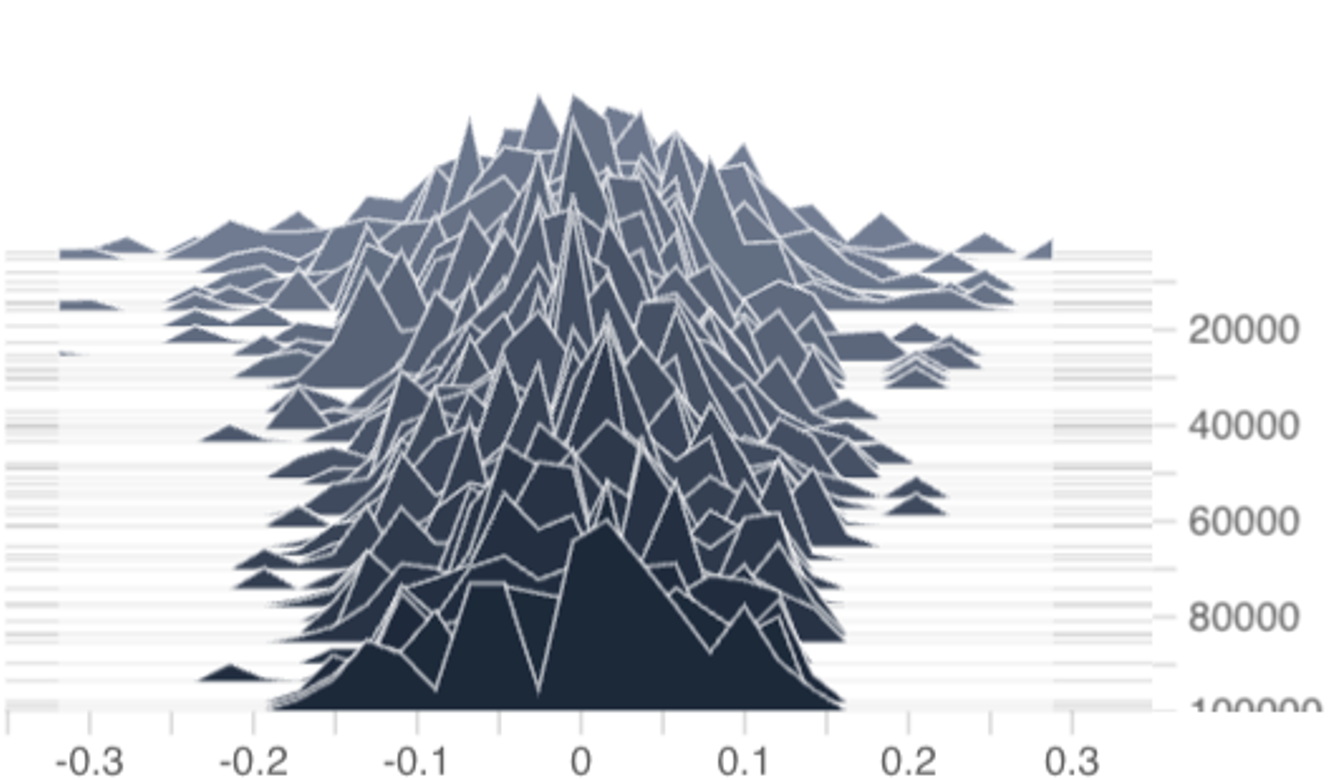}
    \caption{$u/\sqrt{n}$ Step-Size}
  \end{subfigure}
  \caption{Time-lapse histograms of the marginal distributions of the ALF-LB biases $p$ during the training of \mbox{1B-parameter} DeepSeekMoE models using different choices of step-size (Section \ref{sec:primal_dual_alflb}). No explicit constraints were enforced on $p$. Section \ref{sec:experiments} provides experimental details.\label{fig:bias_histograms}}
\end{figure}

Additionally, we will make the technical assumption that $\mathrm{diam}(p) := \max_j p_j - \min_j p_j \le 1 - \kappa$ for some constant $\kappa > 0$, which we found holds without explicit enforcement in our experiments on 1B-parameter DeepSeekMoE models (Figure \ref{fig:bias_histograms}). Thus, we can realistically assume
\[p \in \operatorname{dom_{\calZ}^\kappa} := \{p\in \calZ: \operatorname{diam}(p)\le 1-\kappa\}.\]

\subsection{Strong Convexity}\label{sec:strong_convexity}

Next, observe that the component densities of the affinity scores $\Gamma$ are continuous and strictly positive on $(0,1)$. Thus, by the continuity of $w^{(K)}_{k\ell}(p)$ in $p$ and compactness of $\operatorname{dom_{\calZ}^\kappa}$,
\[
  c_{K}(d) \,\,:=\,\, \inf_{\substack{p\in \operatorname{dom_{\calZ}^\kappa}}}\,\min_{k<\ell} w^{(K)}_{k\ell}(p) \,>\, 0.
\]
Hence, by Proposition \ref{prop:second_directional},
\begin{equation}
  \delta^{\top} \nabla^2 \mathbf{F}_K(p)\,\delta \;\ge\; c_{K}(d) \sum_{k<\ell} (\delta_k - \delta_\ell)^2.
  \label{eq:dFd}
\end{equation}
The assumption \eqref{eq:projKassumption} ensures that $p^{(n)} \in \calZ$ for all $n$. Thus, since $\calZ$ is a linear subspace,
\begin{align*}
  p^{(n+1)} &= \proj_{\calZ}\left(p^{(n)} - \delta^\prime\right) \\
   &= \proj_{\calZ}\left(p^{(n)}\right) - \proj_{\calZ}\left(\delta^\prime\right) \\
   &= p^{(n)} - \underbrace{\proj_{\calZ}\left(\delta^\prime\right)}_{\delta}.
\end{align*}
Since the update direction $\delta$ lies in $\calZ$,
\[
  \sum_{k<\ell} (\delta_k-\delta_\ell)^2 \;=\; E\,\|\delta\|^2 - \left(\sum_{k=1}^E \delta_k\right)^2 =\; E\,\|\delta\|^2.
\]
Combining with property \eqref{eq:dFd}, this yields
\[
  \delta^{\top} \nabla^2 \mathbf{F}_K(p)\,\delta \;\ge\; c_{K}(d) E \,\|\delta\|^2.
\]
Recall the expected loss is $\mathbf{f}(p) = T\mathbf{F}_K(p) - L\sum_{k} p_k$ and observe the linear term does not affect curvature; thus, for all $p,\delta \in \calZ$ with $p$ having diameter at most $d$, $\mathbf{f}$ is $\mu_K$-strongly convex with
\begin{equation}
  \mu_K := T c_{K}(d) E.
  \label{eq:muTCE}
\end{equation}

\subsection{Logarithmic Regret Bound for ALF-LB}\label{sec:log_regret}
Consider the minimizer of the expected loss
\[
p^* = \arg\min_{p\in\calZ} \mathbf{f}(p).
\]
Since $\mathbf{f}$ is $\mu_K$-strongly convex in $\calZ$, $p^*$ is necessarily unique.

Define the regret $R_N := \sum_{n=1}^N \left(f^{(n)}(p^{(n)}) - f^{(n)}(p^*)\right)$. We now give a logarithmic bound on the expected regret $\E[R_N]$ with the ALF-LB update
\begin{equation}
  p^{(n+1)} \gets \proj_{\calZ}\left(p^{(n)} - \epsilon^{(n)} \, \nabla f^{(n)}\left(p^{(n)}\right)\right).
  \label{eq:alflb_proj}
\end{equation}
While the details are adapted to the specific problem at hand, the proof technique is standard in online convex optimization (see, for example, \citet[Section 3.3.1]{hazan2016introduction}). For clarity, define the following short-hand notations:
\begin{equation*}
  \Delta_n := \E\!\left[\|p^{(n)}{-}p^*\|^2\right], \; s_n := \sum_{k=1}^E \pi_k\!\left(p^{(n)}\right)^2, \; a_n := \E\left[\mathbf{f}\left(p^{(n)}\right)-\mathbf{f}\left(p^*\right)\right], \; \sigma^2_{T,E,K} := T^2\left(K-\tfrac{K^2}{E}\right).
\end{equation*}

\begin{lem}[One-step accounting] \label{lem:per_round}
Under the assumptions and notations of Section \ref{sec:oco_notation}-\ref{sec:strong_convexity}, for any $\epsilon^{(n)}{>}0$, the iteration \eqref{eq:alflb_proj} satisfies
\begin{equation}
  2\,a_n \;\le\; \frac{\Delta_n-\Delta_{n+1}}{\epsilon^{(n)}} \;{-}\; \mu_K\,\Delta_n \;+
  \epsilon^{(n)}\,\sigma^2_{T,E,K}.
  \label{eq:per_round_grad}
\end{equation}
\end{lem}
\begin{proof}
Since $\calZ$ is a linear subspace, the projection operator is nonexpansive. Thus,
\begin{align*}
  \|p^{(n+1)}{-}p^*\|^2 &= \left\|\proj_{\calZ}\left(p^{(n)}-\epsilon^{(n)}\nabla f^{(n)}\left(p^{(n)}\right)\right)-p^*\right\|^2 \\
  & \le \left\|p^{(n)}-\epsilon^{(n)}\nabla f^{(n)}\left(p^{(n)}\right)-p^*\right\|^2 \\
  & \le\; \|p^{(n)}{-}p^*\|^2 
  \,{-}\, 2\epsilon^{(n)}\!\left\langle \nabla f^{(n)}\left(p^{(n)}\right),\, p^{(n)}{-}p^*\right\rangle +\, \left(\epsilon^{(n)}\right)^2\,\left\|\nabla f^{(n)}\left(p^{(n)}\right)\right\|^2.
\end{align*}
Taking conditional expectation and using Proposition \ref{prop:unbiased} gives
\begin{multline*}
  \E\left[\|p^{(n+1)}{-}p^*\|^2\mid p^{(n)}\right]
  \le \|p^{(n)}{-}p^*\|^2 - 2\epsilon^{(n)}\!\left\langle \nabla \mathbf{f}(p^{(n)}), p^{(n)}{-}p^*\right\rangle \\
  + (\epsilon^{(n)})^2\,\E\left[\,\|\nabla f^{(n)}(p^{(n)})\|^2\mid p^{(n)}\right].
\end{multline*}
Since the TopKInd decision is invariant to adding the same constant to all coordinates of $p$, we can assume without loss of generality that $p^* \in \calZ$. Thus, since $\calZ$ is a linear subspace, $p^{(n)}-p^* \in \calZ$. Then, the $\mu_K$-strong convexity of $\mathbf{f}$ in $\calZ$ (Section \ref{sec:strong_convexity}) gives
\[
  2\left(\mathbf{f}\left(p^{(n)}\right)-\mathbf{f}\left(p^*\right)\right) + \mu_K\,\|p^{(n)}{-}p^*\|^2 \le 2\left\langle \nabla \mathbf{f}(p^{(n)}), p^{(n)}{-}p^*\right\rangle.
\]
Combining the last two expressions, taking total expectation, and rearranging gives
\begin{equation}
  2\,a_n \;\le\; \frac{\Delta_n-\Delta_{n+1}}{\epsilon^{(n)}} \;{-}\; \mu_K\,\Delta_n \;+
  \epsilon^{(n)}\,\E\!\left[\,\|\nabla f^{(n)}(p^{(n)})\|^2\right].
  \label{eq:per_round_basic}
\end{equation}
Recall from Section \ref{sec:oco_notation} that the gradient is $\nabla f^{(n)}(p)=A^{(n)}(p)-L\,\ones$ where $\sum_k A^{(n)}_k(p)=TK$ and each $A^{(n)}_k(p)\in[0,T]$. It is then easy to check that, for any $p$,
\[
  \|\nabla f^{(n)}(p)\|^2 \le \sigma^2_{T,E,K}.
\]
Substituting this bound into \eqref{eq:per_round_basic} yields the desired result.
\end{proof}
\begin{thm}\label{thm:logarithmic_regret}\textbf{(Logarithmic Regret)}
Consider the update \eqref{eq:alflb_proj} run for $N$ iterations with $\epsilon^{(n)}=1/(\mu_K n)$. Then,
\[
  \mathbb{E}[R_N] \;\le\; \frac{\sigma^2_{T,E,K}}{2\mu_K}\,(1+\ln N).
\]
\end{thm}
\begin{proof} Observe that $\E[R_N]=\sum_{n=1}^N a_n$. Summing \eqref{eq:per_round_grad} over $n=1, \dots, N$ gives
\[
    2\sum_{n=1}^N a_n \le \sum_{n=1}^N \left(\frac{\Delta_n-\Delta_{n+1}}{\epsilon^{(n)}} - \mu_K\Delta_n\right) + \sum_{n=1}^N \epsilon^{(n)} \sigma^2_{T,E,K}.
\]
The first term on the right-hand side is a telescoping sum which evaluates to
\begin{align*}
   \sum_{n=1}^N \left(\frac{\Delta_n-\Delta_{n+1}}{\epsilon^{(n)}} - \mu_K\Delta_n\right)
   &= \sum_{n=1}^N \left(\mu_K n (\Delta_n-\Delta_{n+1}) - \mu_K\Delta_n\right) \\
   &= \mu_K \sum_{n=1}^N \left((n{-}1)\Delta_n - n\Delta_{n+1}\right)\\
   &= -\,\mu_K\,N\,\Delta_{N+1}.
\end{align*}
Dropping this non-positive term, we are left with
\[
  2\sum_{n=1}^N a_n \;\le\; \sum_{n=1}^N \epsilon^{(n)}\,\sigma^2_{T,E,K}
  \;=\; \frac{\sigma^2_{T,E,K}}{\mu_K}\sum_{n=1}^N \frac{1}{n}.
\]
Invoking the classic $\sum_{n=1}^N \tfrac{1}{n} \le 1{+}\ln N$ inequality and dividing by 2 yields the desired result.
\end{proof}

\paragraph{Acknowledgements.} The authors thank Alexandre Belloni, John Birge, Rene Caldentey, Chamsi Hssaine, and Nian Si for helpful discussions and feedback during the preparation of this paper. The authors also thank the Booth School of Business, University of Chicago for financial support that enabled this research. The empirical experiments and models training described in this paper were conducted on Chicago Booth's Pythia Supercomputer Cluster with the assistance of the Chicago Booth IT Department to whom the authors are grateful.

\FloatBarrier

\bibliography{main}
\bibliographystyle{plainnat}

\end{document}